\documentclass{amsart}

\usepackage[utf8]{inputenc}
\usepackage[english]{babel}
\usepackage{amsfonts,amsmath,amssymb,stmaryrd}
\usepackage{faktor}
\usepackage{amsthm}
\usepackage{mathabx}
\usepackage{url}
\usepackage{hyperref}

\usepackage{MnSymbol}
\newcommand{\indep}[4]{#1\underset{#2}{\downfree^{#4}}#3}
\newcommand{\notIndep}[4]{#1\underset{#2}{{\not\downfree}^{#4}}#3}
\def\mod{\textup{ mod }}
\theoremstyle{plain}

\newtheorem{theorem}{Theorem}[section]
\newtheorem{proposition}[theorem]{Proposition}
\newtheorem{fact}[theorem]{Fact}
\newtheorem{lemma}[theorem]{Lemma}
\newtheorem{corollary}[theorem]{Corollary}

\theoremstyle{remark}
\newtheorem{remark}[theorem]{Remark}%[definition]
\theoremstyle{definition}
\newtheorem{definition}[theorem]{Definition}
\newcommand{\tp}{\textup{tp}}

\newcommand{\acl}{\textup{acl}}
\newcommand{\dcl}{\textup{dcl}}
\newcommand{\RV}{\textup{RV}}
\renewcommand{\subset}{\subseteq}
\renewcommand{\supset}{\supseteq}
\renewcommand{\epsilon}{\varepsilon}

\begin{document}
%\section{Introduction}
\author{Akash HOSSAIN}\thanks{The author was partially supported by GeoMod AAPG2019 (ANR-DFG), Geometric and Combinatorial Configurations in Model Theory.}\address{Département de mathématiques d'Orsay, Université Paris-Saclay, France}\email{akash.hossain@universite-paris-saclay.fr}
\title{Extension bases in Henselian valued fields}
\date{}
\maketitle

\begin{abstract}
We study the behaviour of forking in valued fields, and we give several sufficient conditions for parameter sets in a Henselian valued field of residue characteristic zero to be an extension base. Notably, we consider arbitrary (potentially imaginary) bases, whereas previous related results in the literature only focus on maximally complete sets of parameters. This enables us in particular to show that forking coincides with dividing in (the imaginary expansions of) the ultraproducts of the $p$-adic fields.
\end{abstract}
%\tableofcontents
%\pagebreak

%!TEX root = main.tex

\section{Introduction}
The elementary class of Henselian valued fields of residue characteristic zero has been an important subject of study in model theory since the fundamental results of Ax-Kochen and Ershov. As any field of characteristic zero can be the residue field of a structure from this class, one can note that the complete theories of these valued fields can have an arbitrarily “bad" complexity with respect to the standard dividing lines of pure model theory (for instance, one can look at the trivially-valued field $\mathbb{Q}$).
\par Forking and dividing are abstract independence notions from pure model theory that have deep interactions with those dividing lines. The independence theorems for stable and simple theories \cite{kimPillay} are perhaps the best examples of such interactions. Of course, we cannot expect these notions to behave very well in any theory outside of the simple world. For example, one very basic property expected of forking independence fails in general: a tuple $c$ might not be independent from $A$ over $A$. However, one common pattern that has emerged from work on forking in wider classes like NIP or NTP$_2$ is that, over bases avoiding this pathological behaviour (called extension bases), some results from stability and simplicity theory can be pushed forward (as in \cite{hrushovskiPillay}, \cite{KapCherNTP2}, \cite{yaacov_chernikov_2014}, for instance). One crucial missing ingredient to be able to use the forking machinery in unstable theories is therefore to identify extension bases.
\par To this day, as far as we know, the results from the literature about forking in Henselian valued fields only hold in the context where the base is a maximally-complete field, or even a maximally-complete model (see for instance \cite{HHM}, chapter 13). This paper is an attempt to begin a more elaborate study of this subject, with no assumptions on the base.
\par Forking in general is not easy to understand in valued fields, especially in immediate extensions (thus the maximal completeness hypothesis to get rid of this case), but the good understanding that we have on unary types in Henselian valued fields allows us to have some control over forking for unary types. By a standard induction argument using left-transitivity (see Proposition \ref{reducK2}), our study of forking in dimension one translates to results on extension bases in any dimension. Recall that a parameter set $A$ is an extension base when we have $\indep{c}{A}{A}{f}$ for every finite tuple $c$ (of some $|A|^+$-saturated model containing $A$). Let us state the main results of this paper.

\paragraph{\textit{Necessary hypothesis}} Let $K$ be a valued field with respective value group and residue field $\Gamma$, $k$. A necessary condition for every set in $K$ to be an extension base is that it is already the case in $\Gamma$, $k$, thus we assume that every set is an extension base in the respectives theories of the pure ordered group $\Gamma$, and the pure field $k$. We also need to assume that $[k^*:(k^*)^N]$ is finite for every $N>0$, an hypothesis that also appears in \cite{ealy2022residue} and \cite{Chernikov2016HENSELIANVF}. This assumption is strictly weaker than the Galois group of $K$ being bounded, which it is a very common condition in the model theory of fields.

Our main theorem is as follows:

%\begin{theorem}\label{introMain1}
%Let $K$ be a Henselian valued field of residue characteristic zero. Suppose the following conditions hold:\\
%$C^0$ and [($C^2_\Gamma$ and $C^1_k$) or ($C^1_\Gamma$ and $C^2_k$)]
%
%\par Then, there exists $X$ a $\emptyset$-type-definable set in a countable number of variables from $K$ so that any subset of $K$ containing at least one realization of $X$ is an extension base.
%\end{theorem}
%
%More precisely, a realization of $X$ is the reunion of:
%\begin{itemize}
%\item For each $\alpha\in \faktor{k^*}{(k^*)^N}$, a pullback of $\alpha$ in $\mathcal{O}_K$.
%\item For each $\gamma\in\Gamma(\dcl(\emptyset))$, a pullback of $\gamma$ in $K$.
%\end{itemize}
%with $k$, $\Gamma$ the respective residue field and value group of $K$. For a more formal statement, see \ref{main1} and \ref{main1Bis}.
%\par In the above theorem, we need to add countably many points to our sets to obtain extension bases. We have a theorem which does not require this operation:

\begin{theorem}\label{introMain2}
Let $K$ be a Henselian valued field of residue characteristic zero satisfying the assumptions of the above paragraph. If nonforking coincides with field-theoretic algebraic independence in the theory of the pure field $k$, then any subset of $K$ is an extension base.% (and \ref{passageNTP2} applies).
\par Moreover, if $\Gamma\equiv \mathbb{Z}$, then any subset of the imaginary expansion $K^{eq}$ is an extension base.
\end{theorem}

\begin{corollary}
Forking (with arbitrary imaginary parameters) coincides with dividing, and Lascar strong types coincide with Kim-Pillay strong types, in the theory of any non-principal ultraproduct of the $\mathbb{Q}_p$.
\end{corollary}

\begin{proof}
It is a well-known fact that any such theory has NTP$_2$. In particular, by (\cite{KapCherNTP2}, Theorem 1.2) and (\cite{yaacov_chernikov_2014}, Corollary 3.6), forking coincides with dividing and Lascar strong types coincide with Kim-Pillay strong types over extension bases.
\end{proof}

Note that by (\cite{CHERNIKOV2014695}, Theorem 7.6), the above corollary also applies to any field of Laurent series over any PAC bounded field of characteristic $0$.
\par We also prove in this paper more technical results which extend the main theorem to a much weaker context, at the cost of having to name countably many constants. See Theorems \ref{main1}, \ref{main1Bis} for precise statements.

\par Let us explain the outline of the paper.
\par Section \ref{prelims} contains various well-known results about Henselian valued fields, their model theory, and their geometry. We also introduce an ad-hoc notion of genericity and weak-genericity for chains of balls.
\par Let us explain the content of section \ref{sectStPl}. In a Henselian valued field of residue characteristic zero $K$ with residue field $k$, one can show that any $K$-definable subset $X$ of $k$ is $k$-definable in the pure field structure of $k$. We say in this case that $k$ is \textit{stably-embedded}. However, if $X$ is $A$-definable with $A$ a subfield of $K$, $X$ might not be $res(K(\acl(A)))$-definable. This stronger property, which we will call \textit{strong stable embeddedness}, does not always hold, and we will give some sufficient conditions for it to hold (see Lemma \ref{stPlsyntaxique}). This result is interesting on its own, and it could have other applications than the study of extension bases. We will also give an example of a Henselian valued field of residue characteristic zero where the residue field is not strongly stably embedded (in this case we will talk about \textit{weak stable embeddedness}). We prove the results of this section by playing with some automorphisms of Hahn fields that are elementary extensions of our model. A more conventional way to carry out these proofs would be to use results on the model theory of expansions of short exact sequences of Abelian groups (see \cite{suitesExactes}, subsection 4.1). We chose to take a different approach, and provide a proof that is more self-contained, and arguably uses less technical tools.
\par In section \ref{sectForking}, we give the basic definitions of forking, dividing and extension bases, as well as some properties that are needed for our work.
\par In section \ref{sectMain}, we start by proving key lemmas on unary forking in Henselian valued fields, and then we prove the main theorems.

%!TEX root = main.tex

\section{Preliminaries on Henselian valued fields}\label{prelims}

\subsection{Notations}
Let $HF_{0, 0}$ be the theory of Henselian valued fields of residue characteristic zero. If $\mathcal{M}$ is a model of this theory, then we shall write $K(\mathcal{M})$ (resp. $\Gamma(\mathcal{M})$, $k(\mathcal{M})$) for the domain of the valuation of $\mathcal{M}$ (resp. its value group, its residue field). We will write $\mathcal{O}_\mathcal{M}$ for the valuation ring of $\mathcal{M}$, and $\mathfrak{M}_\mathcal{M}$ for its maximal ideal. Let $\RV(\mathcal{M})$ be the Abelian group defined as $\faktor{K(\mathcal{M})^*}{1+\mathfrak{M}_\mathcal{M}}$. When there is no ambiguity, we may omit to specify the model and simply write $K$, $\Gamma$, $k$, $\mathcal{O}$, $\mathfrak{M}$, $\RV$. We may also interpret them as $\emptyset$-definable set in the theory of $\mathcal{M}^{eq}$. The Abelian group $\Gamma$ will always be written additively. We will extend the valuation map to a map of domain $K$ by adding an element $\infty$ to $\Gamma$ for the image of zero. Similarly, we shall also add $\infty$ to $\RV$ to extend the domain of the map $K^*\longrightarrow\RV$ to $K$. Let us write $\RV_\infty$ for the monoïd $\RV, \infty$, and define $\Gamma_\infty$ similarly. The canonical maps $K\longrightarrow\Gamma_\infty$ and $\RV_\infty\longrightarrow\Gamma_\infty$ will both be called $val$. As for the maps $K\longrightarrow k$ and $K\longrightarrow\RV_\infty$, we will respectively call them $res$ and $rv$. Let us define on $\RV_\infty$ the following $\emptyset$-definable ternary predicate:
$$\oplus (x, y, z)\ :\ \exists x', y'\in K\ \left(rv(x')=x\wedge rv(y')=y\wedge rv(x'+y')=z\right)$$
\par We have a short exact sequence of Abelian groups given by:
$$
1\longrightarrow \mathcal{O}^*\longrightarrow K^*\longrightarrow\Gamma\longrightarrow 0
$$
Whenever this sequence splits, we will write $ac$ (resp. $s$) for the induced group homomorphism $K^*\longrightarrow\mathcal{O}^*\longrightarrow k^*$ (resp. $\Gamma\longrightarrow K^*$). These maps do not always exist in a valued field. Even if they do, they are not necessarily definable, and in general they do not restrict (resp. extend) to maps that make the sequence split in an elementary substructure (resp. extension). They are not part of the structures and the languages we consider, unless we explicitly say so. If case these maps exist, the expansion of the valued field to $s$ and $ac$ shall be called a Pas field.
\par We will write $k'((t^G))$ for the valued field of all Hahn series with coefficients in the field $k'$, and powers in the ordered Abelian group $G$. This valued field is known to be Henselian. The above sequence splits in $k'((t^G))$, with $s(\gamma)=t^\gamma$, and $ac(S)$ the coefficient of least value in $S$.
\par In general, if $S$ is a (potentially imaginary) sort (like $K$, $\Gamma$, $\RV$), and $A$ is a parameter subset of $\mathcal{M}^{eq}$, we shall write $S(A)$ for the parameter set $S(\mathcal{M})\cap A$. The non-forking (resp. non-dividing) independence relation will be denoted $\indep{}{}{}{f}$ (resp. $\indep{}{}{}{d}$), whereas the field-theoretic algebraic independence shall be written $\indep{}{}{}{alg}$. The model-theoretic algebraic closure of a parameter set $A$ will be denoted $\acl (A)$, whereas the field-theoretic algebraic closure of a field $F$ will be written $F^{alg}$. The divisible closure of any torsion-free Abelian group $G$ will be denoted $\mathbb{Q}\otimes G$.

\begin{remark}\label{precisionNotations}
If $\mathcal{M}$ is a valued field, and $A\subset \mathcal{M}^{eq}$, there might exist an element of $k(A)$ (ie an element of $k(\mathcal{M})\cap A$, remember that $k$ is the residue field sort, ie an $\emptyset$-definable subset of $\mathcal{M}^{eq}$) that cannot be pulled back to an element of $K(A)$. In other words, $k(A)$ and $res(K(A))$ might not coincide, these two notations are not aliases.
\par The notations $k$, $K$, $\Gamma$ refer to definable sets. When we deal with an actual field or group, we will usually call them $k'$, $G$ as in the above definition of Hahn fields.
\end{remark}

\begin{remark}\label{extensionAC}
Let $\mathcal{M}$ be a Pas field, and $A\leqslant B$ subfields of $K(\mathcal{M})$. If $ac(A)=res(A)$, and $val(A)=val(B)$, then it is easy to show that $ac(B)=res(B)$:
\par Let $b\in B^*$. Let $a\in A$ so that $val(a)=val(b)$. By hypothesis, we can find $a'\in \mathcal{O}_A$ so that $res(a')=ac(a)$. Then we have $ac(b)=res\left(\frac{b}{a}\cdot a'\right)\in res(B)$.
\end{remark}

\subsection{Model-theoretic facts}

\begin{proposition}[\cite{flenner}, Propositions 3.6, 4.3 and 5.1]  \label{flenner}
Let $\mathcal{M}\models HF_{0, 0}$, and $A\subset\mathcal{M}$ so that $A=K(A)^{alg}\cap K(\mathcal{M})$. Let $n, m<\omega$, and $X$ an $A$-definable subset of $K^n\times\RV_\infty^m$. Then $X$ is definable by a formula using only parameters from $K(A)$, the language of the field $K$ without equality, the language of the monoïd $\RV_\infty$ (with equality), the map $rv$, and the relation $\oplus$, without any quantifiers in $K$. Moreover, if $n=1$ and $m=0$, then we can assume that the $K$-variable $x$ in the formula only occurs in terms of the form $rv(x-a)$, with $a\in K(A)$.
\end{proposition}

\begin{proof}
Proposition 4.3 of the paper by Flenner is exactly the first part of the proposition without our hypothesis $A=A^{alg}\cap K(\mathcal{M})$.
\par In the proof of Proposition 5.1 of the same paper, the $\alpha_i$ are given by applying Proposition 3.6 to polynomials with coefficient in $A$. The statement of this Proposition 3.6 shows that the $\alpha_i$ can be chosen as roots in $K(\mathcal{M})$ of some iterated derivatives of these polynomials, so they belong to $A^{alg}\cap K(\mathcal{M})$, which concludes the proof.
\end{proof}

\begin{corollary}
Let $\mathcal{M}\models HF_{0, 0}$, and $A\subset K(\mathcal{M})$. Then $K(\acl(A))=A^{alg}\cap K(\mathcal{M})$.
\end{corollary}

\begin{proof}
Let $X$ be a finite $A$-definable subset of $K$. The proposition gives us a formula defining $X$ with parameters in $A^{alg}\cap K(\mathcal{M})$, which satisfies the second part of the proposition. Let $(a_i)_i$ be the finite family of points from $A^{alg}\cap K(\mathcal{M})$ that appear in this formula. Suppose towards contradiction that $X$ has a point $a$ distinct from all the $a_i$. Let $\gamma\in\Gamma(\mathcal{M})$ be strictly larger than all the $val(a-a_i)$. The set
$$
Y=\left\lbrace b\in K(\mathcal{M})|val(a-b)\geqslant\gamma\right\rbrace
$$
is infinite, but all its elements belong to $X$, for if $b\in Y$, then we have $rv(b-a_i)=rv(a-a_i)$ for all $i$. As a result, $X$ is infinite, a contradiction.
\end{proof}

\begin{remark}
If $\mathcal{M}\models HF_{0, 0}$, then the above corollary implies in particular that $K(\acl(\emptyset))\subset\mathbb{Q}^{alg}$, so this field is trivially-valued.
\end{remark}

\begin{lemma}\label{valeursLibres}
Let $\mathcal{M}\models HF_{0, 0}$, and $a_1...a_n, b_1...b_n\in K(\mathcal{M})$. Suppose \\$(val(a_1),...val(a_n))$ is a $\mathbb{Q}$-free family in $\mathbb{Q}\otimes\Gamma(\mathcal{M})$, and $rv(a_i)=rv(b_i)$ for each $i$. Then $a_1...a_n\equiv_{\emptyset}b_1...b_n$.
\end{lemma}

\begin{proof}
By \ref{flenner}, we just have to prove $rv(P(a_1...a_n))=rv(P(b_1...b_n))$ for each polynomial $P$ with $\emptyset$-definable coefficients. Let $P$ be such a polynomial. By the above remark, the non-zero coefficients of $P$ all have value $0$. As the $val(a_i)$ form a $\mathbb{Q}$-free family, the values of every monomial involved in $P(a_1...a_n)$ are pairwise-distinct. Let $Q$ be the monomial in $P$ so that $Q(a_1...a_n)$ has least value. Then we have $rv(P(a_1...a_n))=rv(Q(a_1...a_n))$. As $rv(a_i)=rv(b_i)$, we have $val(a_i)=val(b_i)$, so $Q(b_1...b_n)$ is also the monomial of $P(b_1...b_n)$ of least value, and $rv(P(b_1...b_n))=rv(Q(b_1...b_n))$. Now, $rv(Q(a_1...a_n))$ is just the product of powers of the $rv(a_i)$ by a constant. By hypothesis $rv(a_i)=rv(b_i)$, we have $rv(Q(a_1...a_n))=rv(Q(b_1...b_n))$. That concludes the proof.
\end{proof}

\begin{fact}[\cite{Cherlin1976ModelTA}, Corollary 28]\label{existencePasExtension}
If $\mathcal{M}\models HF_{0, 0}$ is $\aleph_1$-saturated, then $s$ and $ac$ exist.
\end{fact}

\begin{proposition}[see for instance \cite{vandenDries2014}, Corollary 5.25]\label{stablementPlonge}
Let $\mathcal{M}\models HF_{0, 0}$, and $A\subset\mathcal{M}$ so that $A=\acl(K(A))$. Then any $A$-definable subset of $\Gamma(\mathcal{M})$ is actually $val(A)$-definable in the ordered group structure $\Gamma(\mathcal{M})$. Moreover, any $A$-definable subset $X$ of $k(\mathcal{M})$ is actually $k(\mathcal{M})$-definable in the ring structure $k(\mathcal{M})$, and, if $ac$ exists in $\mathcal{M}$, then $X$ is $ac(A)$-definable in the ring structure $k(\mathcal{M})$.
\end{proposition}

\begin{corollary}\label{dominationValeurCorpique}
Let $\mathcal{M}\models HF_{0, 0}$, $A\subset \mathcal{M}$, and $\gamma$, $\delta\in \Gamma(\mathcal{M})$, so that $A=\acl(K(A))$. If $\gamma\equiv_{\Gamma(A)}\delta$ in the ordered group $\Gamma(\mathcal{M})$, then $\gamma\equiv_A\delta$ in the valued field $\mathcal{M}$.
\end{corollary}

The reader will find hints of the proof of the following lemma in (\cite{zeta}, proofs of Proposition 5.9 and Lemma 5.10, and Remark 5.14). An explicit proof is provided in section \ref{sectStPl} of this paper.

\begin{lemma}\label{zetaRV}
Let $\mathcal{M}\models HF_{0, 0}$, and $X$ a definable subset of $\RV$. Then there exists a finite family $\gamma_1...\gamma_n\in \Gamma$, definable subsets $X_i\subset \RV$ of $val^{-1}(\gamma_i)$, and a bound $N<\omega$, so that $X_i\subset X$ for all $i$, and $X\setminus\left(\bigcup\limits_i X_i\right)$ is a reunion of cosets of the $\emptyset$-definable subgroup of $\RV$ corresponding to the embedding:
$$
(k^*)^N\longrightarrow k^*=\faktor{\mathcal{O}^*}{1+\mathfrak{M}}\longrightarrow \faktor{K^*}{1+\mathfrak{M}}=\RV
$$
\end{lemma}

The above lemma should also follow from the material on short exact sequences in \cite{suitesExactes}.

\subsection{Balls}

\begin{definition}
Let $\mathcal{M}$ be a valued field. A \textit{ball} $X$ in $\mathcal{M}$ is a definable subset of $K(\mathcal{M})$ of the form:
$$
X=\left\lbrace x\in K|val(x-c) R \gamma\right\rbrace
$$
with $c\in K(\mathcal{M})$, $\gamma\in\Gamma(\mathcal{M})$, and $R$ is either $>$ (in that case $X$ is \textit{open}) or $\geqslant$ (here $X$ is \textit{closed}). If $X$ is defined as above, then its radius is $rad(X)=\gamma$. If $A\subset\mathcal{M}^{eq}$, then $X$ is an \textit{$A$-ball} if the imaginary canonical parameter of $X$ is an element of $\dcl^{eq}(A)$. Moreover, we say that an $A$-ball $X$ is \textit{pointed} when it has a point in $K(A)$.
\par We will define two notions: genericity and weak genericity. Let us first give an illustration before the formal details. Suppose $A=\acl^{eq}(A)\subset\mathcal{M}^{eq}$, and $c\in K(\mathcal{M})$. The set $\mathfrak{B}$ of every $A$-ball that contains $c$ turns out to be a chain with respect to inclusion, this is a well-known fact in ultrametric geometry. Now, not every $A$-ball is necessarily pointed, so let $\mathfrak{B}'$ be the chain of pointed $A$-balls containing $c$. The chain $\mathfrak{B}'$ is a final segment of $\mathfrak{B}$ which might be strict. Then, the type-definable set $\cap\mathfrak{B}'$ is an approximation of $c$ that is weaker than $\cap\mathfrak{B}$. Even if this approximation is weaker, it is witnessed by points from $K(A)$ (because the balls of $\mathfrak{B}'$ are pointed), hence it is often easier to manipulate than $\cap\mathfrak{B}$, especially when $A$ is generated by field parameters.
\par Let us now define formally our notions of genericity. Let $\mathfrak{B}$ be a chain (with respect to inclusion) of $A$-balls, and $c\in K$ in an elementary extension of $\mathcal{M}$. By convention, if $\mathfrak{B}=\emptyset$, then $\cap\mathfrak{B}$, the intersection of every ball in $\mathfrak{B}$, is the definable set $K$. We say that $c$ is \textit{$A$-generic} of $\mathfrak{B}$ if the following conditions hold:
\begin{itemize}
\item $c$ is in the $A$-type-definable set $\cap\mathfrak{B}$.
\item For every $A$-ball $X\subsetneq\cap\mathfrak{B}$, we have $c\not\in X$.
\end{itemize}
The above conditions clearly yield a partial type over $A$. We can use basic saturation arguments to show that such a type is consistent. In case $\mathfrak{B}$ has a least element, then we have to use the fact that the residue field is infinite (and maybe use a density argument in case the least ball is not closed). Else, consistency is simply established by compactness.
\par Likewise, if $\mathfrak{B}$ is a chain of \textbf{pointed} $A$-balls, then $c$ is said to be \textit{weakly $A$-generic} of $\mathfrak{B}$ when the following conditions hold:
\begin{itemize}
\item $c\in\cap\mathfrak{B}$.
\item For every \textbf{pointed} $A$-ball $X\subsetneq\cap\mathfrak{B}$, we have $c\not\in X$.
\end{itemize}
\par Note that our notion of weak genericity only makes sense when $\mathfrak{B}$ is a chain of pointed $A$-balls, otherwise it is undefined. If $\mathfrak{B}$ has a non-pointed element, then an $A$-generic point of $\mathfrak{B}$ is weakly $A$-generic of some chain of $A$-balls whose intersection is strictly coarser than $\cap\mathfrak{B}$, and it is not weakly $A$-generic of $\mathfrak{B}$. However, if every ball of $\mathfrak{B}$ is pointed, then $A$-genericity also implies weak $A$-genericity.
\par Geometrically speaking, the partial type of the $A$-generics of $\mathfrak{B}$ can be written as an intersection of some decreasing sequence of $A$-definable crowns (a crown being a ball from which we removed some subball that can be empty). In case of weak genericity, less crowns are involved, so the set of realizations is larger.
\par If $\mathfrak{B}$ has a least element which is a closed ball, then we will say that $\mathfrak{B}$ is \textit{residual}. If not, then $\mathfrak{B}$ will be called \textit{ramified} (with respect to $A$) if $\cap\mathfrak{B}$ strictly contains an $A$-ball or a point in $K(A)$, and $\mathfrak{B}$ will be called \textit{immediate} (with respect to $A$) otherwise.
\par If $X$ is a ball, $a\in X$ and $b\not\in X$, then we have $rv(b-a)=rv(b-a')$ for all $a'\in X$. As a result, we can define $rv(b-X)$ to be $rv(b-a)$ for any $a\in X$.
\end{definition}

\begin{remark}
If the value group is discrete, then every ball is both open and closed. If it is dense however, then there is no ball that is both open and closed.
\end{remark}

\begin{remark}\label{intersectionsEgales}
\par Let $\mathfrak{B}$ and $\mathfrak{B}'$ be chains of $A$-balls. One can note that the sets of $A$-generic points of each chain of $A$-balls are the classes of some equivalence relation (to belong to the same $A$-balls), which is even $A$-type-definable. As a result, $\mathfrak{B}$ and $\mathfrak{B}'$ have the same $A$-generic points if and only if they have at least one $A$-generic point in common.
\par Using the definition of genericity, one can easily show that $\mathfrak{B}$ and $\mathfrak{B}'$ have the same $A$-generic points if and only if the $A$-type-definable sets $\cap\mathfrak{B}$ and $\cap\mathfrak{B}'$ coincide.
\par Moreover, if the balls of $\mathfrak{B}$ and $\mathfrak{B}'$ are pointed, then the same two characterizations hold when we look at their $A$-weakly generic points.
\end{remark}

When the parameter set is algebraically closed, and generated by field elements, we have a good understanding of when the notions of genericity and weak genericity do not coincide:

\begin{proposition}\label{fortVSFaible}
Let $\mathcal{M}\models HF_{0, 0}$, $A\subset \mathcal{M}^{eq}$ so that $A=\acl^{eq}(K(A))$, and $\mathfrak{B}$ a chain of pointed $A$-balls. Suppose there is $c\in K(\mathcal{M})$, $A$-weakly-generic of $\mathfrak{B}$, but not $A$-generic of $\mathfrak{B}$.
\par Then $\mathfrak{B}$ is residual. Moreover, if $Y$ is the largest open ball that contains $c$ which is strictly contained in $\cap\mathfrak{B}$, then $Y$ is an $A$-ball (which must not be pointed by weak genericity of $c$), and $c$ is $A$-generic of $\{Y\}$.
\end{proposition}

\begin{proof}
By hypothesis, there exists $Y$ an $A$-ball containing $c$ that is not pointed. In particular, $Y$ is strictly contained in $\cap\mathfrak{B}$. By \ref{flenner}, there exist a finite family $(a_i)_{i<N}$ of elements of $K(A)$, and $Z$ an $A$-definable subset of $\RV^N$, so that:
$$
Y=\left\lbrace x\in K|(rv(x-a_i))_i\in Z\right\rbrace
$$
Clearly, as $Y$ is not pointed, none of the $a_i$ belongs to $Y$. Let $\gamma\in\Gamma(A)$ be the least of the $val(a_i-Y)$. There does not exist any $\delta\in\Gamma(\mathcal{M})$ so that $\gamma<\delta<rad(Y)$, for if not, then choose $b\in K(\mathcal{M})$ for which $val(b-c)=\delta$, for every $i$ we have $rv(a_i-b)=rv(a_i-c)$, so $b\in Y$, which contradicts $val(b-c)<rad(Y)$.
\par If $\Gamma(\mathcal{M})$ is dense, this must imply that $Y$ is open and $rad(Y)=\gamma$. If $\Gamma(\mathcal{M})$ is dicrete, then of course $Y$ is open and closed, and the radius of $Y$ as a closed ball is the successor of $\gamma$.
\par As $Y$ is open, let $X$ be the least closed ball that strictly contains $Y$. As $X$ is definable over the canonical parameter of $Y$, $X$ is an $A$-ball. The radius of $X$ is $\gamma$, so at least one of the $a_i$ must be in $X$, so $X$ is pointed. There is no ($\mathcal{M}^{eq}$-)ball between $X$ and $Y$, and $c\in Y$, so $c$ must be $A$-weakly generic of $\{X\}$. Now, by \ref{intersectionsEgales} applied to $\{X\}$ and $\mathfrak{B}$, $X$ must be a lower bound of $\mathfrak{B}$. Then, $X$ must be the least element of $\mathfrak{B}$, otherwise a standard compactness argument could be used to show that the inclusion of type-definable sets $X\subset\cap\mathfrak{B}$ would be strict, which would contradict \ref{intersectionsEgales}. As a result, we notice that $X$ is canonical, its construction does not depend on the choice of $Y$, but $Y$ does depend on $X$: it is the largest open ball containing $c$ and strictly contained in $X$. Therefore, the choice of $Y$ must be unique: $Y$ is the only $A$-ball containing $c$ that is not pointed. There is no $A$-ball smaller than $Y$ containing $c$, so $c$ is $A$-generic of $\{Y\}$, this concludes the proof.
\end{proof}

\begin{corollary}\label{relevementCentre}
Let $\mathcal{M}\models HF_{0, 0}$, and $A\subset \mathcal{M}^{eq}$ so that $A=\acl^{eq}(K(A))$. If $\Gamma(\mathcal{M})$ is dense, then any closed $A$-ball is pointed.
\end{corollary}

\begin{proof}
Let $X$ be a closed $A$-ball, and $c\in K(\mathcal{M})$ an $A$-generic point of $\{X\}$. Let $\mathfrak{B}$ be the set of every pointed $A$-ball containing $c$. If $X$ was not pointed, then $c$ would not be $A$-generic of $\mathfrak{B}$, so we could apply the proposition to show that $c$ would be $A$-generic of $\{Y\}$ for some open ball $Y$. However, $X$ is not open, so we reach a contradiction by \ref{intersectionsEgales} applied to $\{X\}$ and $\{Y\}$, and $X$ must be pointed.\end{proof}

\begin{corollary}
Let $\mathcal{M}\models HF_{0, 0}$, and $A\subset \mathcal{M}^{eq}$ so that $A=\acl^{eq}(K(A))$. Let $\mathfrak{B}$ be a chain of $A$-balls. If $\mathfrak{B}$ is ramified, then $\cap\mathfrak{B}$ actually has a point in $K(A)$.
\end{corollary}

\begin{proof}
Let $Y$ be an $A$-ball strictly contained in $\cap\mathfrak{B}$. If $Y$ is closed and $\Gamma(\mathcal{M})$ is dense, then we are done by the corollary.
\par Else, let $X$ be the least closed ball strictly containing $Y$ (which is in fact the least ball strictly containing $Y$). The ball $X$ is clearly an $A$-ball, and we have $X\subset\cap\mathfrak{B}$ by minimality of $X$. This inclusion is strict as $\mathfrak{B}$ is not residual. Let $c$ be an $A$-generic point of $\{Y\}$. If $Y$ is not pointed, then we can apply the proof of the proposition (to the set of pointed $A$-balls containing $c$) to show that $X$ is pointed, so $\cap\mathfrak{B}$ has a point in $K(A)$ either way.
\end{proof}

\begin{proposition}\label{zetaBall}
Let $\mathcal{M}\models HF_{0, 0}$, and $A\subset\mathcal{M}^{eq}$ so that $A=\acl^{eq}(K(A))$. Let $c,d\in K(\mathcal{M})$ be singletons, and $\mathfrak{B}$ the set of pointed $A$-balls that contain $c$. Suppose $d\in\cap\mathfrak{B}$. Then:
\begin{itemize}
\item If $\cap\mathfrak{B}$ has no point in $K(A)$, then $c\equiv_A d$.
\item If $a\in K(A)$ is in $\cap\mathfrak{B}$, and $rv(c-a)=rv(d-a)$, then $c\equiv_A d$.
\end{itemize}
\end{proposition}

\begin{proof}
As $A=\acl^{eq}(K(A))$, it is sufficient to show $c\equiv_{K(A)}d$. By \ref{flenner}, it is sufficient to show that $rv(c-a')=rv(d-a')$ for all $a'\in K(A)$. Let $a'\in K(A)$.
\begin{itemize}
\item Suppose $a'\not\in \cap\mathfrak{B}$. Let $X\in\mathfrak{B}$ so that $a'\not\in X$. Then $rv(a'-c)=rv(a'-X)=rv(a'-d)$.
\item Suppose $a'\in\cap\mathfrak{B}$ (so the point $a$ from the statement of the proposition exists). Let $X$ be the least closed ball that contains $a$ and $a'$. Then  $X$ is a pointed $A$-ball that does not contain $c$, so $rv(c-a')=rv(c-a)$ and $val(c-a')<rad(X)$. As $rv(d-a)=rv(c-a)$, we have $val(d-a)=val(c-a)<rad(X)$, so $val(d-a')<rad(X)$, so $d\not\in X$, which implies $rv(d-a')=rv(d-a)=rv(c-a)=rv(c-a')$.
\end{itemize}
Either way, we have $rv(c-a')=rv(d-a')$, so we can conclude.
\end{proof}

An extensive study of unary types over imaginary parameters has been carried out in \cite{zeta}, in the setting of pseudo-local fields. In particular, there is an analogue of the above proposition with imaginary parameters, when the value group is a $\mathbb{Z}$-group:

\begin{proposition}[\cite{zeta}, Lemma 3.7]\label{zetaBall2}
Let $\mathcal{M}\models HF_{0, 0}$ with $\Gamma(\mathcal{M})\equiv\mathbb{Z}$, and $A\subset\mathcal{M}^{eq}$ so that $A=\acl^{eq}(A)$. Let $c,d\in K(\mathcal{M})$ be singletons, and $\mathfrak{B}$ the set of $A$-balls that contain $c$. Suppose $d\in\cap\mathfrak{B}$, and $\mathfrak{B}$ is $A$-immediate. Then $c\equiv_A d$.
\end{proposition}

\begin{proposition}\label{consistanceGenerique}
Let $\mathcal{M}\models HF_{0, 0}$, $A\subset B\subset\mathcal{M}^{eq}$ so that $A=\acl^{eq}(A)$, $\mathfrak{B}$ a chain of $A$-balls, and $c\in K(\mathcal{M})$ a point $A$-generic of $\mathfrak{B}$. Then there exists $c'\equiv_A c$ (in an elementary extension of $\mathcal{M}$) that is $B$-generic of $\mathfrak{B}$.
\end{proposition}

\begin{proof}
Suppose not. Then, by compactness, there exists a finite family $(X_i)_i$ of $B$-balls strictly contained in $\cap\mathfrak{B}$ so that, for all $c'\equiv_A c$, we have $c'\in\bigcup\limits_i X_i$. Replace $\mathcal{M}$ by a $|B|^+$-saturated, strongly $|B|^+$-homogeneous elementary extension. \begin{itemize}
\item Suppose $\mathfrak{B}$ is residual. Let $Y_i$ be the largest open ball containing $X_i$, and strictly contained in $min(\mathfrak{B})$ (the least element of $\mathfrak{B}$). The definable set $\bigcup\limits_i Y_i$ is $A$-invariant, so each $Y_i$ is an $\acl^{eq}(A)=A$-ball, so $c\not\in Y_i$ by $A$-genericity. However, we have $c\in\bigcup\limits_i Y_i$, a contradiction.
\item Suppose $\mathfrak{B}$ is not residual. Without loss of generality, suppose each $X_i$ contains at least one $A$-conjugate of $c$. Let $X$ be the least closed ball that contains all the $X_i$. Then $X$ is the least closed ball that contains every $A$-conjugate of $c$, so $X$ is $A$-invariant, therefore $X$ is an $A$-ball. As $c\in X$, $X\in\mathfrak{B}$ (we can replace $\mathfrak{B}$ by the set of every $A$-ball containing $\cap\mathfrak{B}$, without loss of generality). By hypothesis, $X$ is not the least element of $\mathfrak{B}$, so the largest open ball containing $c$ and strictly contained in $X$ (let us call it $Y$) is an $A$-ball. This contradicts the minimality of $X$, as each $X_i$ is strictly contained in $Y$, so the least closed ball that contains all of them must be (weakly) contained in $Y$.\qedhere
\end{itemize}
\end{proof}

\subsection{Lifts of the residue field}

\begin{definition}
In a valued field $\mathcal{M}$, a \textit{lift of the residue field} is a ring morphism $k(\mathcal{M})\longrightarrow \mathcal{O}_\mathcal{M}$ that is a section of $res$.
\end{definition}

In the Hahn field $k'((t^G))$, the standard lift maps an element of $k'$ to the corresponding constant polynomial.

\begin{proposition}[\cite{vandenDries2014}, subsection 2.4]\label{extensionLift}
Let $\mathcal{M}\models HF_{0, 0}$. Then there is a canonical 1-1 correspondence between the lifts of the residue field in $\mathcal{M}$ and the maximal proper subfields of $\mathcal{O}_\mathcal{M}$.
\par Moreover, if $\mathcal{N}\models HF_{0, 0}$, $f\ :\ \mathcal{N}\longrightarrow\mathcal{M}$ is an embedding, and $l$ is a lift of the residue field in $\mathcal{N}$, then there exists at least one lift of the residue field $l'$ in $\mathcal{M}$ so that $l'$ extends $f^{-1} l f\ :\ f(k(\mathcal{N}))\longrightarrow\mathcal{O}_\mathcal{M}$.
\end{proposition}

\begin{corollary}
In any valued field $\mathcal{M}\models HF_{0, 0}$, the residue field can be lifted.
\end{corollary}

\begin{proof}
Apply the proposition to the (canonical) embedding into $\mathcal{M}$ whose domain is the trivially valued field $\mathbb{Q}$ (which is of course Henselian).
\end{proof}

\begin{remark}
If $\mathcal{M}\models HF_{0, 0}$, $a\in K(\mathcal{M})$, and $\alpha\in k(\mathcal{M})$, then the following conditions are equivalent:
\begin{itemize}
\item There exists a lift of the residue field sending $\alpha$ to $a$.
\item $\alpha$ and $a$ have the same ideal over $\mathbb{Q}$.
\end{itemize}
\end{remark}

\begin{lemma}\label{relevementRayon1}
Let $\mathcal{M}\models HF_{0, 0}$. Suppose $ac$ and $s$ exist in $\mathcal{M}$. Let $l$ be a lift of the residue field. Let $a\in K(\mathcal{M})^*$ so that $val(a)\neq 0$. Then $a\equiv_\emptyset l(ac(a))s(val(a))$.
\end{lemma}

\begin{proof}
Let $b=l(ac(a))s(val(a))$. We have $ac(a)=ac(b)$ and $val(a)=val(b)$, so $rv(a)=rv(b)$ (as $ac$ and $s$ exist, we have a group isomorphism between $\RV(\mathcal{M})$ and $k(\mathcal{M})^*\times\Gamma(\mathcal{M})$). Moreover, as $val(a)\neq 0$, we can apply \ref{valeursLibres} to get $a\equiv_\emptyset b$.
\end{proof}

\begin{proposition}\label{relevementRayon2}
Let $\mathcal{M}\models HF_{0, 0}$, and $A=\acl^{eq}(K(A))\subset\mathcal{M}^{eq}$. Suppose the following conditions hold:
\begin{itemize}
\item $ac$ and $s$ exist in $\mathcal{M}$.
\item For every $\gamma\in\Gamma(\dcl(\emptyset))$, there exists an $a\in K(A)$ so that $val(a)=\gamma$.
\item For every $N>0$, $\alpha\in ac(K(A))$, there exists $a\in \mathcal{O}^*_A$ so that \\$res(a)\mod (k^*)^N=\alpha\mod (k^*)^N$.
\end{itemize}
Then, for every $\delta\in\Gamma(A)$, there exists an $a\in K(A)$ so that $val(a)=\delta$.
\end{proposition}

In particular, the radius of every $A$-ball can be written $val(a)$ for some $a\in K(A)$.

\begin{proof}
Let $\delta\in\Gamma(A)$. Then $\delta$ is $K(A)$-definable, so by \ref{stablementPlonge} it is actually $val(K(A))$-definable. By (\cite{QEOAG}, Corollary 1.10), the definable closure of $val(K(A))$ in the ordered group $\Gamma(\mathcal{M})$ is exactly the relative divisible closure of the subgroup of $\Gamma(\mathcal{M})$ generated by $val(K(A))$ and $\Gamma(\dcl(\emptyset))$. In other words, there exist $\gamma\in\Gamma(\dcl(\emptyset))$, $a\in K(A)$, and $N>0$ so that $\delta=\frac{val(a)+\gamma}{N}$. By hypothesis, there exists $a'\in K(A)$ so that $val(a')=\gamma$.
\par Let $\bar{a}\in\mathcal{O}^*_A$ so that $res(\bar{a})\mod (k^*)^N=ac(aa')\mod (k^*)^N$. Let $a''=aa'\bar{a}^{-1}$. We have $a''\in K(A)$ and $val(a'')=val(a)+\gamma$. We will conclude the proof by showing that $a''$ has an $N$th-root in $K(\mathcal{M})$, this root will have value $\delta$, and it will also belong to $K(A)$, as $A=\acl(A)$.
\par We have $ac\left(a''\right)=\frac{ac(aa')}{res(\bar{a})}\in (k^*)^N$. Let $r'\in k^*$ so that $r'^N=ac(a'')$, and let $l$ be a lift of the residue field. We have $(l(r'))^N=l(ac(a''))$, so $l(ac(a''))$ has an $N$th-root in $K(\mathcal{M})$. As a result, $l(ac(a''))s(val(a''))$ also has an $N$th-root in $K(\mathcal{M})$, namely $l(r')s(\delta)$. By \ref{relevementRayon1}, $l(ac(a''))s(val(a''))\equiv_\emptyset a''$, so $a''$ also has an $N$th-root in $K(\mathcal{M})$, and we can conclude.
\end{proof}

%!TEX root = main.tex

\section{The strength of the residue field's stable embeddedness}\label{sectStPl}

In the Proposition \ref{stablementPlonge}, one cannot help but notice that the result on definable subsets of $k$ is weaker than that on $\Gamma$. In particular, we don't have an analogue of \ref{dominationValeurCorpique} in $k$. This property will sometimes be needed for our results on unary forking, so we would like to understand when it holds. In this section, we will show a counterexample in which this analogue does not hold, and we will give sufficient conditions to obtain the same nice behavior as what we have in $\Gamma$. The strategy for the proofs is to work in an elementary extension that is a Hahn field, where the specific structure will give us more tools to build elementary maps.
\par We invite the reader to recall what we said in the Remark \ref{precisionNotations} to not get confused in the notations. For instance, in the Definition \ref{defContreEx}, we have an explicit field $k'$, but we use the notation $k$ to define some definable subset $X$ of $k'$, because $X$ has points in elementary extensions of $k'$ that are not in $k'$.

\subsection{Algebraic technicalities}
In order to build our example and prove our version of \ref{dominationValeurCorpique} in $k$, we first need to state simple field-theoretic facts.

\begin{definition}\label{defAut}
Let $G$ be an ordered Abelian group, and $k'$ a field of characteristic zero. We will define two (injective) group homomorphisms:
$$
\begin{array}{cccc}
aut^1_{k', G}\ :\ & Aut(k') &\longrightarrow & Aut(k'((t^G)))\\
aut^2_{k', G}\ :\ & Hom(G, k'^*) & \longrightarrow & Aut(k'((t^G)))
\end{array}
$$
The automorphisms of $k'((t^G))$ we consider are of course valued field automorphisms.
\par For every $\sigma\in Aut(k')$, $aut^1_{k', G}(\sigma)$ is defined as $\sum\limits_\alpha a_\alpha t^\alpha\longmapsto\sum\limits_\alpha \sigma(a_\alpha)t^\alpha$.
\par For every $\sigma\in Hom(G, k'^*)$, $aut^2_{k', G}(\sigma)$ is defined as $\sum\limits_\alpha a_\alpha t^\alpha\longmapsto\sum\limits_\alpha a_\alpha \sigma(\alpha) t^\alpha$.
\end{definition}

The map $aut^1_{k', G}$ is classical, and the map $aut^2_{k', G}$ is denoted $P$ in the Definition 3.3.5 of the paper \cite{KUHLMANN2022339}, to which we refer the interested reader for a more elaborate study of automorphisms of Hahn fields (and some of their subfields).

\begin{lemma}\label{normalisateurDiv}
Let $G$ be a divisible group, and $g\in G$. Then there exists $\sigma\in Hom(\mathbb{Q}, G)$ so that $\sigma(1)=g$.
\end{lemma}

\begin{proof}
Let $g_1=g$. By induction on $n> 0$, choose $g_{n+2}\in G$ an $(n+2)$th root of $g_{n+1}$ (ie $g_{n+2}^{n+2}=g_{n+1}$). In particular, $g_{n+1}$ is always an $((n+1)!)$th root of $g$. For each $(n, m)\in\mathbb{Z}\times\mathbb{Z}_{>0}$, define $f(n, m)=(g_m)^{n(m-1)!}$. By straightforward, but rather long computations, one can show that $f$ factors through the canonical surjection $\mathbb{Z}\times\mathbb{Z}_{>0}\longrightarrow\mathbb{Q}$, and that the induced map $\mathbb{Q}\longrightarrow G$ is a group homomorphism sending $1$ to $g$.
\end{proof}

\begin{lemma}\label{residuQuadTransc}
Let $k'$ be a field of characteristic zero, and $n>0$. Then, in the field $k'(u)$ ($u$ being any transcendental element), $\frac{u}{u+n}$ is not a square.
\end{lemma}

\begin{proof}
If $\frac{u}{u+n}$ were a square, then there would exist $P$, $Q\in k'[u]$ so that $Q\neq 0$ and $Q^2\cdot u=P^2\cdot (u+n)$. Now, $k'[u]$ is a UFD where $u$ and $u+n$ are coprime and irreducible. As a result, $u$ appears an even amount of times in the factorization of $P^2\cdot (u+n)$ in $k'[u]$, and it appears an odd amount of times in that of $Q^2\cdot u$, so we cannot have $Q^2\cdot u=P^2\cdot (u+n)$.
\end{proof}

\begin{lemma}
Let $k'$ be an algebraically-closed field of characteristic zero, $1\neq\lambda\in\mathbb{Q}_{>0}$, and $P$, $Q$ coprime non-zero polynomials in $k'[u]$. If $\frac{P(u)}{Q(u)}=\frac{P(\lambda u)}{Q(\lambda u)}$ in $k'(u)$, then both $P$ and $Q$ are constant.
\end{lemma}

\begin{proof}
Suppose, say, $P$ is non-constant (the same proof will work if $Q$ is non-constant), and admits a non-zero root $\alpha$. As $\lambda\in\mathbb{Q}^*$, and $k'$ has characteristic zero, the orbit of $\alpha$ under the bijection $h\ :\ x\longmapsto \lambda^{-1} x$ is infinite. Therefore, there exists $n\in\mathbb{Z}$ so that $h^n(\alpha)$ is a root of $P(u)$, but $h^{n+1}(\alpha)$ is not. As $h^n(\alpha)$ is a root of $P(u)$, $h^{n+1}(\alpha)$ is a root of $P(\lambda u)$. As $P(u)Q(\lambda u)=P(\lambda u)Q(u)$, $h^{n+1}(\alpha)$ is a root of $P(u)Q(\lambda u)$, so it is a root of $Q(\lambda u)$. This contradicts the coprimality hypothesis, as $h^n(\alpha)$ is a common root of $P(u)$ and $Q(u)$.
\par Now, $P$ and $Q$ are coprime and do not admit any non-zero root. In particular, $P$ and $Q$ are either both constant, or one of them is constant and the other can be written $\beta X^N$ for some $\beta\in k'^*$, $N>0$. Then, if $P$ and $Q$ are not both constant, it is easy to see that $\frac{P(u)}{Q(u)}\neq\frac{P(\lambda u)}{Q(\lambda u)}$, which concludes the proof.
\end{proof}

\begin{corollary}\label{fracQcqTransc}
Let $k'$ be a field of characteristic zero, $0<N\neq M>0$, and $F\in k'(u)\setminus k$. Then $F(Mu)\neq F(Nu)$.
\end{corollary}

\begin{proof}
We can assume $k'$ is algebraically-closed. We can write $F(Nu)=\frac{P(u)}{Q(u)}$, with $P$, $Q\in k'[u]$ coprime and non-zero (because $F\neq 0$). If we had $F(Mu)=F(Nu)$, then we would have $\frac{P(u)}{Q(u)}=\frac{P(\frac{M}{N} u)}{Q(\frac{M}{N} u)}$, so both $P$ and $Q$ would be constant by the lemma, so $F\in k'$, this is absurd.
\end{proof}

\subsection{An example of weak stable embeddedness}\label{contreEx}

Let us build an example of a Henselian valued field of residue characteristic zero $\mathcal{M}$, a subfield $A< K(\mathcal{M})$, and an $A$-definable set $X\subset k(\mathcal{M})$ so that $X$ is not $k(\acl(A))$-definable (and hence in particular not $res(K(\acl(A)))$-definable). We will also have $\Gamma(\acl(A))>\Gamma(\dcl(\emptyset))$, so that the example is not too degenerate.

\begin{definition}\label{defContreEx}
Let $k'=\mathbb{Q}^{alg}(u)$, $G=\mathbb{Q}$, and $\mathcal{M}=k'((t^G))$. We will write $ac$ for the natural angular component map of $\mathcal{M}$, and we will identify $k'$ with a subfield of $K(\mathcal{M})$ via the natural lift of the residue field. Let $A\subset K(\mathcal{M})$ be the field generated by $u\cdot t$, and $X=u\mod (k^*)^2=\left\lbrace x\in k|x\cdot u^{-1}\in (k^*)^2\right\rbrace$.
\end{definition}

\begin{proposition}
The set $X$ is $A$-definable.
\end{proposition}

\begin{proof}
Consider the $A$-definable set:
$$Y=\left\lbrace x\in k|\exists y\in K\ val(y)=val(u\cdot t)\wedge y\in (K^*)^2\wedge x=res(u\cdot t\cdot y^{-1})\right\rbrace$$
Let us show that $X=Y$.
\par If $x\in X(\mathcal{M})$, then we have $x=u v^2$ for some $v\in k'^*$. We have $x\in Y(\mathcal{M})$ by choosing $y=(v^{-1}\cdot t^{\frac{1}{2}})^2$.
\par Let $x\in Y(\mathcal{M})$, and $y$ be as in the definition of $Y$. Then $ac(y)$ must be in $(k'^*)^2$, and we have $x=u\cdot ac(y)^{-1}$, so $x\in X(\mathcal{M})$.
\par We have $X(\mathcal{M})=Y(\mathcal{M})$, so $X=Y$.
\end{proof}

\begin{proposition}
We have $k(\acl(A))=\mathbb{Q}^{alg}$.
\end{proposition}

\begin{proof}
Let $N>0$. Let $\sigma_N\in Aut(k')$ be defined as $\sigma_N(F(u))=F(N u)$. Apply \ref{normalisateurDiv} to the divisible group $(\mathbb{Q}^{alg})^*$ to find $\sigma'_N\in Hom(\mathbb{Q}, k'^*)$ so that $\sigma'_N(1)=N^{-1}$. Let $\tau_N=aut^2_{k', G}(\sigma'_N)\circ aut^1_{k', G}(\sigma_N)$. Then $\tau_N\in Aut(\mathcal{M}/A)$. By \ref{fracQcqTransc} (with $k'=\mathbb{Q}^{alg}$), for each $F\in k'\setminus \mathbb{Q}^{alg}$, $(\tau_N(F))_N=(F(N\cdot u))_N$ is an infinite family of pairwise-distinct $A$-conjugates of $F$, so $F\not\in \acl(A)$. This concludes the proof.
\end{proof}

\begin{proposition}
The definable set $X$ has infinitely many $\mathbb{Q}^{alg}$-conjugates.
\end{proposition}

\begin{proof}
For each $N<\omega$, let $\sigma_N\in Aut(k'/\mathbb{Q}^{alg})$ be defined as $\sigma_N(F(u))=F(u+N)$, and $\tau_N=aut^1_{k', G}(\sigma_N)$.\footnote{The map $\sigma_N$ is $\mathbb{Q}^{alg}$-invariant, but it is not $A$-invariant anymore.} By \ref{residuQuadTransc}, the $(\tau_N(X))_N$ are all pairwise-distinct.
\end{proof}

As a result the $A$-definable set $X$ is not $k(\acl(A))$-definable, which concludes our example of a Henselian valued field of residue characteristic zero where $k$ is not strongly stably embedded.

\subsection{Elementary embeddings into Hahn fields}
Apart from the proof of \ref{zetaRV}, the content of this subsection is mostly an application of some of the material covered in the notes of van den Dries \cite{vandenDries2014}.

\begin{proposition}[\cite{vandenDries2014}, Corollary 4.29]\label{extImmUnique}
Let $\mathcal{M}$ be any valued field of residue characteristic zero. Let $\sigma_1$, $\sigma_2$ be valued field embeddings with domain $\mathcal{M}$, and $\mathcal{M}_i$ a maximal immediate extension of $\sigma_i(\mathcal{M})$. Then there exists $\tau$ a valued field isomorphism $\mathcal{M}_1\longrightarrow\mathcal{M}_2$ so that $\sigma_2=\tau\circ\sigma_1$.
\end{proposition}

\begin{proposition}\label{plongementHahnSimple}
Let $\mathcal{M}$ be a Henselian Pas field of residue characteristic zero, and $l$ a lift of the residue field. Then there exists a valued field embedding $\sigma\ :\ \mathcal{M}\longrightarrow k(\mathcal{M})((t^{\Gamma(\mathcal{M})}))$ so that the following conditions hold:

\begin{itemize}
\item $\sigma$ is the identity on $k(\mathcal{M})\cup\Gamma(\mathcal{M})$.
\item For every $\alpha\in k(\mathcal{M})$, $\sigma(l(\alpha))$ is the constant series $\alpha$.
\item For every $\gamma\in\Gamma(\mathcal{M})$, $\sigma(s(\gamma))=t^\gamma$.
\end{itemize}
Moreover, $\sigma$ is an elementary map of Pas fields.
\end{proposition}

\begin{proof}
For the construction of $\sigma$ and the proof that it satisfies the conditions from the list, onc can read the discussion from the notes of van den Dries between the lemmas 4.30 and 4.31.
\par For the proof that $\sigma$ is an elementary map, one can read the section 5.5 of the notes, especially the definitions at the beginning and theorem 5.21. It is clear that $\sigma$ is a “good map" with respect to these definitions.
\end{proof}

We will actually need a construction that is a little more specific to prove Lemma \ref{zetaRV}:

\begin{proposition}\label{plongementHahn}
Let $G$ be an ordered Abelian group, $k'$ a field of characteristic zero, and $\mathcal{M}$ a Pas elementary extension of $k'((t^G))$ (in the Pas language). Then there exists $\sigma$ an elementary embedding of Pas fields $\mathcal{M}\longrightarrow k(\mathcal{M})((t^{\Gamma(\mathcal{M})}))$ so that the following conditions hold:
\begin{itemize}
\item $\sigma$ is the identity on $\Gamma(\mathcal{M})\cup k(\mathcal{M})$.
\item The restriction of $\sigma$ to $k'((t^G))$ is the inclusion map.
\end{itemize}
\end{proposition}

\begin{proof}
The proof is written in the same spirit as in the notes of van den Dries.
\par Let $l$ be a lift of the residue field in $\mathcal{M}$ extending that of $k'((t^G))$ (by \ref{extensionLift}). Let $K'$ be some maximal immediate extension of $\mathcal{M}$, and $L$ be the subfield of $\mathcal{M}$ generated by $k'((t^G))\cup s(\Gamma(\mathcal{M}))\cup l(k(\mathcal{M}))$. We have a valued field embedding $\tau\ :\ L\longrightarrow k(\mathcal{M})((t^{\Gamma(\mathcal{M})}))$ sending each $s(\gamma)$ to $t^\gamma$, sending each $l(\alpha)$ to the constant polynomial $\alpha$, and sending each element of $k'((t^G))$ to itself. Now, the extension $L\leqslant\mathcal{M}$ is clearly immediate, so $K'$ and $k(\mathcal{M})((t^{\Gamma(\mathcal{M})}))$ are both maximal immediate extensions of $L$. By \ref{extImmUnique}, we have $\sigma$ a valued field isomorphism $K'\longrightarrow k(\mathcal{M})((t^{\Gamma(\mathcal{M})}))$ that extends $\tau$. By construction of $\tau$, $\sigma$ is also a Pas isomorphism. Notice now that the inclusion $\sigma(\mathcal{M})\longrightarrow k(\mathcal{M})((t^{\Gamma(\mathcal{M})}))$ is a “good map", hence an elementary embedding of Pas fields. As a result, $\mathcal{M}\longrightarrow\sigma(\mathcal{M})\longrightarrow k(\mathcal{M})((t^{\Gamma(\mathcal{M})}))$ witnesses the proposition.
\end{proof}

Now, as promised, here is a proof of Lemma \ref{zetaRV}. The techniques used in this proof are quite similar to those used in subsection \ref{subSectStPlFort}. We give less formal details here than in the next subsection.

\begin{proof}[Proof of Lemma \ref{zetaRV}]
Let $H$ be the $\emptyset$-type-definable group $\bigcap\limits_{N<\omega}(k^*)^N$. For each $N<\omega$, let $X_N$ be the definable set: $$\left\lbrace x\in \RV| x\mod (k^*)^N\subset X\right\rbrace$$ Let $X_\omega=\left\lbrace x\in \RV| x\mod H\subset X\right\rbrace$. We just have to prove that $val(X\setminus X_N)$ is finite for some $N<\omega$. If not, then by compactness $val(X\setminus X_\omega)$ must be infinite. Let us show that this is impossible.
\par We can apply \ref{plongementHahnSimple} to some Pas field that is an elementary extension of $\mathcal{M}$ (which exists by \ref{existencePasExtension}) to find $\mathcal{M}_1$ a Hahn field that is an elementary extension of $\mathcal{M}$. Let $\mathcal{N}$ be a $|\mathcal{M}_1|^+$-saturated elementary extension of $\mathcal{M}_1$. By \ref{plongementHahn}, let $\mathcal{M}_2$ be a Hahn field that contains $\mathcal{M}_1$ so that we have an elementary embedding $\sigma\ :\ \mathcal{N}\longrightarrow\mathcal{M}_2$ which restricts to the identity on $\mathcal{M}_1$.
\par By compactness and saturation, if $val(X\setminus X_\omega)$ was infinite, then there would exist $x\in\RV(\sigma(\mathcal{N}))$ so that $x\in X\setminus X_\omega$ and $val(x)\not\in \Gamma(\mathcal{M}_1)$. Let $y$ be in $H(\sigma(\mathcal{N}))$ so that $xy\not\in X$. By \ref{normalisateurDiv} applied to the divisible group $H$, let $\tau\in Hom(\mathbb{Q}\cdot val(x), H)$ so that $\tau(val(x))=y$. Let $A\leqslant\mathbb{Q}\otimes\Gamma(\mathcal{M}_2)$ so that $\mathbb{Q}\otimes\Gamma(\mathcal{M}_2)=(\mathbb{Q}\otimes\Gamma(\mathcal{M}_1))\oplus A\oplus \mathbb{Q}\cdot val(x)$. Let $\tau'$ be the group homomorphism $\mathbb{Q}\otimes\Gamma(\mathcal{M}_2)\longrightarrow k(\mathcal{M}_2)^*$ that extends $\tau$ which is trivial on $(\mathbb{Q}\otimes\Gamma(\mathcal{M}_1))\oplus A$. Then $aut^2_{k(\mathcal{M}_2), \Gamma(\mathcal{M}_2)}(\tau'_{|\Gamma(\mathcal{M}_2)})$ is a valued field automorphism of $\mathcal{M}_2$ leaving $\mathcal{M}_1$ pointwise-invariant, and sending $x$ to $xy$. As $X$ is $\mathcal{M}_1$-definable, we must have $xy\in X$, a contradiction.
\end{proof}

\subsection{A sufficient condition for strong stable embeddedness}\label{subSectStPlFort}

\begin{lemma}\label{semantiqueFortementStPlonge}
Let $\mathcal{M}\models HF_{0, 0}$ be $\aleph_1$-saturated, and $A\subset\mathcal{M}^{eq}$ an algebraically-closed parameter set. Suppose $\mathcal{M}$ is strongly-$|A|^+$-homogeneous, and let $ac$ be an angular component on $\mathcal{M}$. Let $H$ be the divisible $\emptyset$-type-definable group $\bigcap\limits_{N>0} (k^*)^N$. Suppose for all $r\in ac(K(A))$, there exists $a\in\mathcal{O}_A$ so that $res(a)\mod H=r\mod H$. Then every $A$-definable subset of $k$ is actually $res(K(A))$-definable in the pure field $k(\mathcal{M})$.
\end{lemma}

\begin{proof}
Let $X$ be an $A$-definable subset of $k$. Any element of $\mathbb{Q}\otimes val(K(A))$ is a $\mathbb{Q}$-linear combination of elements from $val(K(A))$, so there exists $(\gamma_i)_i$ a family of values from $val(K(A))$ that is a $\mathbb{Q}$-base of $\mathbb{Q}\otimes val(K(A))$. In other words, $\mathbb{Q}\otimes val(K(A))$ is the direct sum of the Abelian groups $(\mathbb{Q}\gamma_i)_i$. Let $a_i\in K(A)$ so that $val(a_i)=\gamma_i$, and $a'_i\in \mathcal{O}_A$ so that $res(a'_i)\mod H=ac(a_i)\mod H$. Let $r_i=ac\left(\frac{a_i}{a'_i}\right)\in H$.
\par Write for short $k'=k(\mathcal{M})$, $G=\Gamma(\mathcal{M})$, and let us identify $\mathcal{M}$ with a Pas elementary substructure of $k'((t^G))$ by using \ref{plongementHahnSimple}. Note that the Hahn series in $k'((t^G))$ that have finite support belong to $\mathcal{M}$ by the conditions of \ref{plongementHahnSimple}.
\par By \ref{valeursLibres}, and by strong-homogeneity, there exists $\sigma\in Aut(\mathcal{M})$ so that $\sigma\left(\frac{a_i}{a'_i}\right)=r_it^{\gamma_i}$. By \ref{normalisateurDiv} applied to $H$, as $\mathbb{Q}\gamma_i$ is isomorphic to $\mathbb{Q}$, choose $\tau_i\in Hom(\mathbb{Q}\gamma_i, k'^*)$ so that $\tau_i(\gamma_i)=r_i^{-1}$. By the universal property of the direct sum, we can find $\tau'\in Hom\left(\mathbb{Q}\otimes val(K(A)), k'^*\right)$ extending each of the $\tau_i$. Now, choose $B$ a $\mathbb{Q}$-vector-subspace of $\mathbb{Q}\otimes G$ so that $\mathbb{Q}\otimes G=(\mathbb{Q}\otimes val(K(A)))\oplus B$, and extend $\tau'$ to $\tau''\in Hom\left(\mathbb{Q}\otimes G, k'^*\right)$ so that $\tau''_{|B}=1$. Let $\tau=aut^2_{k', G}(\tau''_{|G})$. Then $\tau$ is an automorphism\footnote{The map $\tau$ is an automorphism of valued fields, but it is not an automorphism of Pas fields.} of $k'((t^G))$ leaving $k'$, $G$ pointwise-invariant, and sending each $r_i t^{\gamma_i}$ to $t^{\gamma_i}$. Let $L$ be the subfield of $K(A)$ generated by the $\frac{a_i}{a'_i}$. Then we clearly have $res(\tau\circ\sigma (L))=ac(\tau\circ\sigma (L))$. By \ref{extensionAC}, we have $res(\tau\circ\sigma (K(A)))=ac(\tau\circ\sigma (K(A)))$. By \ref{stablementPlonge}, $\tau\circ\sigma(X)$ is $res(\tau\circ\sigma (K(A)))$-definable in the pure field $k'$. Now, as $\mathcal{M}$ is an elementary substructure, $\sigma$ and hence $\tau\circ\sigma$ is an elementary map, so $X$ is $res(K(A))$-definable in the pure field $k(\mathcal{M})$.
\end{proof}

\begin{lemma}\label{stPlsyntaxique}
Let $\mathcal{M}\models HF_{0, 0}$, and $A\subset\mathcal{M}^{eq}$ so that $A=\acl^{eq}(K(A))$. Suppose $[k^*:(k^*)^N]$ is finite for each $N>0$. Suppose for each $N>0$, $\alpha\in \faktor{k^*}{(k^*)^N}$, there exists $a\in \mathcal{O}_A$ so that $res(a)\mod (k^*)^N=\alpha$. Then any $A$-definable subset of $k$ is $res(K(A))$-definable in the pure field $k(\mathcal{M})$.
\end{lemma}

\begin{proof}
By the finiteness hypothesis, we can freely replace $\mathcal{M}$ by an $\aleph_1$-saturated, strongly-$|A|^+$-homogeneous elementary extension. Let $X$ be an $A$-definable subset of $k$, and $H=\bigcap\limits_{N>0}(k^*)^N$. For each $N>0$, let $\pi_N$ be the canonical projection $\faktor{k^*}{H}\longrightarrow\faktor{k^*}{(k^*)^N}$. Let $p$ be the partial type
$$
\left\lbrace res(x_\alpha)\mod (k^*)^N= \pi_N(\alpha)|\alpha\in \faktor{k^*(\mathcal{M})}{H(\mathcal{M})}, N>0\right\rbrace
$$
The partial type $p$ does not necessarily have a realization in $K(A)$, but it is finitely satisfiable in $K(A)$ by hypothesis on $A$.
\par Let $E_N$ be the set of every formula with parameters in $K(A)$ and variables in $(x_\alpha)_{\alpha\in \faktor{k^*(\mathcal{M})}{H(\mathcal{M})}}$, $y_1...y_N$. For each $\phi\in E_N$ and $m>0$, let $\bar{\phi}_m$ be the formula in the variables $(x_\alpha)_{\alpha\in \faktor{k^*}{H}}$ stating that the set:
$$
\left\lbrace y_1...y_N\in K|\phi((x_\alpha)_\alpha, y_1...y_N)\right\rbrace
$$
has exactly $m$ elements. Let $F_N$ be the set of every formula (in the language of the pure field $k$) without parameters, with variables $y_1...y_N\in k$, $z\in k$. Let $q$ be the following partial type in the variables $(x_\alpha)_\alpha$:
$$
\left\lbrace
\begin{array}{l|l}
\bar{\phi}_m\Longrightarrow[ & N, m>0\\
\forall y_1...y_N\in \mathcal{O}\ (\phi((x_\alpha)_\alpha, y_1...y_N)\Longrightarrow & \phi\in E_N\\
\neg [\forall z\in k\ (z\in X\Longleftrightarrow \psi(res(y_1)...res(y_N), z))])] & \psi\in F_N
\end{array}
\right\rbrace
$$
The realizations of $q$ are the tuples $(a_\alpha)_\alpha$ for which $X$ is not $res(K(\acl(A(a_\alpha)_\alpha)))$-definable. The definition of $q$ can be read “for all $y_i$, if $y_i$ is $A(x_\alpha)_\alpha$-algebraic via the formula $\phi$, then the formula $\psi(res(y_i), z)$ does not define $X$".
\par Now the Lemma \ref{semantiqueFortementStPlonge} implies that the type $p\cup q$ is inconsistent. By compactness, as $p$ is finitely satisfiable in $K(A)$, $K(A)$ has a tuple that is not a realization of $q$. This concludes the proof.
\end{proof}

\begin{corollary}\label{dominationResiduCorpique}
Let $\mathcal{M}\models HF_{0, 0}$, and $A\subset \mathcal{M}$, so that $A=\acl(K(A))$. Suppose $[k^* : (k^*)^N]$ is finite for every $N>0$, and $A$ satisfies the conditions of the lemma. Then, for all $r$, $s\in k(\mathcal{M})$, if $r\equiv_{k(A)} s$ in the field $k(\mathcal{M})$, then we have $r\equiv_A s$ in the valued field $\mathcal{M}$.
\end{corollary}

\begin{remark}
In the example of subsection \ref{contreEx}, we have $u\equiv_{k(A)} u+1$, but $u\not\equiv_A u+1$.
\end{remark}

The lemma also allows us to prove an analogue of \ref{relevementCentre}. It will be useful for one of our main results (Theorem \ref{main1Bis}).

\begin{corollary}\label{relevementRV}
With the same hypothesis as the lemma, suppose additionally:
\begin{itemize}
\item For every $\gamma\in\Gamma(\dcl(\emptyset))$, there exists $a\in K(A)$ so that $val(a)=\gamma$.
\item In the pure field $k(\mathcal{M})$, for any subfield $B$ for which $B^{alg}\cap k(\mathcal{M})=B$, we have $B=\acl(B)$ in the theory of the field $k(\mathcal{M})$.
\end{itemize}
Then, for any $\alpha\in \RV(A)$, there exists $a\in K(A)$ so that $rv(a)=\alpha$.
\end{corollary}

\begin{proof}
Let $\alpha\in \RV(A)$. By \ref{relevementRayon2}, any value of $\Gamma(A)$ can be pulled back to $K(A)$, so we can assume by scaling everything that $val(\alpha)=0$, ie $\alpha\in k^*$. Now, $\alpha\in\acl^{eq}(K(A))$, so by the lemma we actually have $\alpha\in \acl(res(K(A)))$ in the pure field structure of $k(\mathcal{M})$. By hypothesis, $\alpha\in res(K(A))^{alg}\cap k(\mathcal{M})$. Therefore, we have a polynomial $P$ with non-zero coefficients in $\mathcal{O}_\mathcal{M}\cap K(A)$ so that $\alpha$ is a root of $res(P)$. If we choose such a $P$ as a pullback of the minimal polynomial of $\alpha$ over $res(K(A))$, then, by Henselianity and as we are in residue characteristic zero, there exists a root of $P$ in $K(\mathcal{M})$ having residue $\alpha$. As $P$ is chosen whith coefficients in $K(A)$, and $A=\acl^{eq}(K(A))$, this root is in $K(A)$, which concludes the proof.
\end{proof}

\begin{corollary}\label{relevementCentre2}
With the same hypothesis as the above corollary, every $A$-ball is pointed.
\end{corollary}

\begin{proof}
Let $Y$ be an $A$-ball. If $Y$ is not open, then $\Gamma$ is dense, so we can apply \ref{relevementCentre} to conclude. Else, by using \ref{flenner} with a reasoning similar to \ref{relevementCentre}, we can prove that $X$ the minimal closed ball strictly containing $Y$ is pointed. Let $a\in K(A)\cap X$. We can use the above corollary to find $a'\in K(A)$ so that $rv(a')=rv(Y-a)$, and we can conclude as $a'+a\in Y$.
\end{proof}

\section{Forking and dividing}\label{sectForking}

\subsection{Basic facts}
We will explicitly state each property of forking and dividing that will be needed in this paper. For a more general overview of these notions, we refer to \cite{Tent2012ACI} and \cite{Casanovas2011SimpleTA}.
\par Let us recall the definition of forking and dividing. In order to simplify the definitions, we will fix $\kappa$ an arbitrarily large infinite cardinal, and $\mathcal{M}$ a $\kappa$-saturated, strongly-$\kappa$-homogeneous structure. A set $A$ will be called \textit{small} if $|A|<\kappa$.

\begin{definition}
Let $A$ be a small subset of $\mathcal{M}$, and $X$ a definable subset of $\mathcal{M}$ in a small number of variables (this number can be infinite, it does not matter as only a finite number of those variables can appear in the formula defining $X$). We say that $N<\omega$, and $(\sigma_n)_{n<\omega}\in Aut(\mathcal{M}/A)^\omega$ constitute a \textit{witness for division} of $X$ over $A$ when $\sigma_0=id$, and, for all $P\in\mathcal{P}(\omega)$, if $|P|=N$, then $\bigcap\limits_{n\in P}\sigma_n(X)=\emptyset$. We say that $X$ \textit{divides} over $A$ when such a witness for division exists. We say that $X$ \textit{forks} over $A$ if there exists a finite family $X_1...X_n$ of definable subsets of $\mathcal{M}$ so that $X\subset\bigcup\limits_i X_i$, and $X_i$ divides over $A$ for all $i$.
\par We shall say that a partial type $p$ \textit{divides} (resp. \textit{forks}) over $A$ if some finite intersection of definable subsets from $p$ divides (resp. forks) over $A$.
\par Let $B, C$ be small subsets of $\mathcal{M}$. We note $\indep{C}{A}{B}{d}$ (resp. $\indep{C}{A}{B}{f}$) if $tp(C/AB)$ does not divide (resp. does not fork) over $A$.
\end{definition}

\begin{remark}\label{bonnesExtensions}
If $\pi$ is a partial type over $AB$ that does not fork over $A$, then there exists a completion of $\pi$ over $AB$ that does not fork over $A$.
\par This extension property might not hold for dividing, in fact the purpose of the definition of forking is to satisfy this property. Forking can be defined as the subrelation of $\indep{}{}{}{d}$ satisfying this extension property, that is “maximal" in some sense (see \cite{Casanovas2011SimpleTA}, Proposition 12.14).
\par However, by going from dividing to forking, we lose a crucial property: not all sets are extension bases anymore (this notion will be formally defined soon). Therefore, the study of extension bases that we carry out in this paper is a way to understand when forking keeps the good properties that hold for dividing, when we can get the best of both worlds.
\end{remark}

\begin{remark}
\begin{itemize}
\item These definitions do not depend on the choice of the monster model $\mathcal{M}$. In other words, if $\mathcal{N}$ is a $\kappa$-saturated, strongly-$\kappa$-homogeneous structure, $A'B'C'\subset\mathcal{N}$, and $ABC\longmapsto A'B'C'$ is a partial elementary map with respect to $\mathcal{M}$ and $\mathcal{N}$, then we have $\indep{C}{A}{B}{f}\Longleftrightarrow\indep{C'}{A'}{B'}{f}$.
\item If $\notIndep{C}{A}{B}{f}$, then there exists $c_1...c_n\in C$ so that $\notIndep{c_1...c_n}{A}{B}{f}$ (we call that property \emph{left finite character}). 
\item If $X\subset Y$ are two definable sets, and $Y$ divides (resp. forks) over $A$, then so does $X$.
\end{itemize}
\end{remark}

\begin{proposition}\label{divisionPreimage}
If $f$ is an $A$-definable function, and $Y$ is a definable subset of the image of $f$, then $Y$ divides over $A$ if and only if $f^{-1}(Y)$ does.
\end{proposition}

\begin{proof}
Let $(\sigma_n)_n\in Aut(\mathcal{M}/A)^\omega$. As $f$ is $A$-definable, we have the equality:
$$\sigma_n(f^{-1}(Y))=f^{-1}(\sigma_n(Y))$$
Now, if $P\subset\omega$, then:
$$\bigcap\limits_{n\in P}\sigma_n(f^{-1}(Y))=\bigcap\limits_{n\in P}f^{-1}(\sigma_n(Y))=f^{-1}\left(\bigcap\limits_{n\in P}\sigma_n(Y)\right)$$
As $Y$ and its $A$-conjugates are subsets of the image of $f$, such an intersection is empty if and only if $\bigcap\limits_{n\in P}\sigma_n(Y)=\emptyset$. As a result $(\sigma_n)_n$ is a witness for division of $Y$ over $A$ if and only it is also a witness for division of $f^{-1}(Y)$ over $A$, which concludes the proof.
\end{proof}

\begin{fact}\label{forkingEZ}
Let $A$, $B$, $C$ be small subsets of $\mathcal{M}$. Then:
\begin{enumerate}
\item $\indep{C}{A}{B}{f}\Longleftrightarrow \indep{C}{\acl(A)}{B}{f}$. (This is because every $A$-indiscernible sequence is $\acl(A)$-indiscernible.)
\item If $B$ is an $|AC|^+$-saturated, strongly-$|AC|^+$-homogeneous elementary substructure of $\mathcal{M}$ that contains $A$, then $\indep{C}{A}{B}{f}\Longleftrightarrow\indep{C}{A}{B}{d}$. Moreover, if $p$ is a complete type over $B$ in a finite number of variables, and the orbit of $p$ under $Aut(B/A)$ is smaller than $|A|^+$ (in particular, if this orbit is a point), then $p$ does not divide over $A$. (See for instance \cite{Tent2012ACI}, exercises 7.1.3 and 7.1.5.)
\item If $D$ is a small subset of $\mathcal{M}$, $\indep{C}{A}{B}{f}$, and $\indep{D}{AC}{B}{f}$, then $\indep{CD}{A}{B}{f}$ (we call that \emph{left transitivity}).
\end{enumerate}
\end{fact}

Let us conclude this subsection by an example set in the class of ordered Abelian groups, which will be useful later.

\begin{remark}\label{coupure}
Let $G$ be a non-trivial ordered Abelian group, and $A, B\subset G$. Suppose the following conditions hold:
\begin{itemize}
\item $G$ is $|AB|^+$-saturated.
\item $\forall a\in A\ \forall b\in B\ a<b$.
\item $G$ is dense or $A$ does not have a largest element.
\end{itemize}
Then, one can show by an easy compactness argument that the type-definable set $X=\bigcap\limits_{a\in A,b\in B}\left]a, b\right[$ is non-empty in $G$.
\par Moreover, suppose $G\equiv\mathbb{Z}$, $A'$, $B'$ are small subsets of $G$ so that $A\subset A'\subset B'\supset B$, and $X$ does not have any point in $\dcl(B')$. Then one can prove that $\indep{c}{A'}{B'}{f}$. This can be done by looking at the unary types in the ordered group $\mathbb{Z}$, which are easy to understand. In fact, more generally, the author has soon to be available preprint about forking in a larger class of ordered Abelian groups, from which this  result follows immediatly.
\end{remark}

\subsection{Further results}
The facts in this subsection are already known, but we do not know of good references where they are proved explicitly. Propositions \ref{leftAcl} and \ref{rightAcl} are proved in (\cite{KapCherNTP2}, Lemma 3.21.(2)) with hypothesis that are not necessary.

\begin{lemma}\label{fibresFinies}
Let $A$ be a small subset of $\mathcal{M}$, $Z$, $Z'$ $A$-definable sets, $R\subset Z\times Z'$ an $A$-definable set, and $m<\omega$. Suppose that, for all $x\in Z$, the set $R_x=\left\lbrace y\in Z'|R(x, y) \right\rbrace$ has exactly $m$ elements.
\par Then, if $Y\subset Z'$ is an $\mathcal{M}$-definable set that divides over $A$, then the definable set $X=\left\lbrace x\in Z|\exists y\in Y\ R(x, y)\right\rbrace$ divides over $A$.
\end{lemma}

\begin{proof}
By contraposition, suppose $X$ does not divide over $A$. Let $N<\omega$, and $(\sigma_n)_{n<\omega}\in Aut(\mathcal{M}/A)^\omega$ so that $\sigma_0=id$. As $X$ does not divide over $A$, there must exist $P$ a subset of $\omega$ of size $Nm$ so that $\bigcap\limits_{n\in P}\sigma_n(X)\neq\emptyset$. Let $\alpha$ be in that intersection. As $\alpha\in Z$, let $\beta_1...\beta_m$ be the elements of $R_\alpha$. For each $n\in P$, let $1\leqslant f(n)\leqslant m$ so that $\beta_{f(n)}\in \sigma_n(Y)$. By the pigeonhole principle, there exists $1\leqslant q\leqslant m$ so that $q$ has at least $N$ antecedents from $f$. Let $Q=\left\lbrace n\in P|f(n)=q\right\rbrace$. Then $Q$ is a subset of $\omega$ of size $\geqslant N$ for which $\beta_q\in\bigcap\limits_{n\in Q}\sigma_n(Y)$, hence $N, (\sigma_n)_n$ is not a witness for division of $Y$ over $A$. This holds for all $N, (\sigma_n)_n$, so $Y$ does not divide over $A$, and we can conclude.
\end{proof}

\begin{proposition}\label{leftAcl}
Let $A$, $B$, $C$, $D$ be small subsets of $\mathcal{M}$ so that $\acl(AC)\subset\acl(AD)$. If $\indep{D}{A}{B}{f}$, then $\indep{C}{A}{B}{f}$.
\end{proposition}

\begin{proof}
Let $c$ be a finite tuple from $C$, $Y$ an $AB$-definable set that contains $c$, $p<\omega$, and $(Y_i)_{0\leqslant i\leqslant p}$ some $\mathcal{M}$-definable sets so that $Y\subset\bigcup\limits_i Y_i$. Let $\phi(x, y)$ be a formula with parameters in $A$ so that $\phi(D, y)$ is the least finite (say, of size $m<\omega$) $AD$-definable set that contains $c$. Let $\psi(x)$ be the $A$-definable formula stating that there exists exactly $m$ distinct points $y_1..y_m$ for which $\phi(x, y_i)$ holds for all $i$. Let $X$ be the $AB$-definable set containing (some enumeration of) $D$ defined by the following formula with free variable $x$:
$$\psi(x)\wedge\left(\exists y\ \left(\phi(x, y)\wedge y\in Y\right)\right)$$
Define similarly $X_i$ from $Y_i$. By hypothesis, $X$ does not fork over $A$. We clearly have $X\subset\bigcup\limits_i X_i$, so there exists $j$ so that $X_j$ does not divide over $A$. By \ref{fibresFinies}, $Y_j$ does not divide over $A$. This holds for every choice of a finite family $(Y_i)_i$, so $Y$ does not fork over $A$. This holds for every $Y$, so $\indep{c}{A}{B}{f}$. This holds for every finite tuple $c$ of $C$, so by left finite character we have $\indep{C}{A}{B}{f}$.
\end{proof}

\begin{proposition}\label{rightAcl}
If $A$, $B$, $C$ are small subsets of $\mathcal{M}$, and $\notIndep{C}{A}{\acl(AB)}{f}$, then $\notIndep{C}{A}{B}{f}$.
\end{proposition}

Note that the other direction is an easy consequence of the definition of forking.

\begin{proof}
Let $X$ be an $\acl(AB)$-definable set containing $C$ that forks over $A$. Let $X_1...X_n$ witness that $X$ forks over $A$, ie $X\subset\bigcup\limits_i X_i$ and each $X_i$ divides over $A$. Now, $X$ has finitely many $AB$-conjugates, so there exist $N<\omega$ and $\sigma_1...\sigma_N\in Aut(\mathcal{M}/AB)$ so that $\left\lbrace \sigma_j(X)|1\leqslant j\leqslant N\right\rbrace$ is the orbit of $X$ under the action of $Aut(\mathcal{M}/AB)$. The reunion $Y=\bigcup\limits_j \sigma_j(X)$ is an $AB$-definable set containing $C$. As each $\sigma_j$ pointwise-fixes $A$, each $\sigma_j(X_i)$ divides over $A$ (apply $\sigma_j$ to a witness for division). As $Y\subset\bigcup\limits_{ij}\sigma_j(X_i)$, we have $\notIndep{C}{A}{B}{f}$, which concludes the proof.
\end{proof}

\begin{lemma}\label{divisionBase}
Let $A$, $C$ be small subsets of $\mathcal{M}$, and $X$ a definable subset of $\mathcal{M}$ in a small number of variables that divides over $A$. Suppose $(2^{|AC|})^+<\kappa$ (an assumption that can always be made as $\kappa$ is arbitrarily large, and we can replace $\mathcal{M}$ by a sufficiently large elementary extension). Then there exists $C'\equiv_A C$, $C'\subset \mathcal{M}$, so that $X$ divides over $AC'$.
\end{lemma}

\begin{proof}
Let $b\in\mathcal{M}$ so that $X$ is $b$-definable. By a standard compactness argument, from a witness for division of $X$ over $A$, we can find a sequence $(\sigma_i)_{i<(2^{|AC|})^+}\in Aut(\mathcal{M}/A)^{(2^{|AC|})^+}$, $N<\omega$, so that, for all $P\subset(2^{|AC|})^+$, if $|P|=N$, then $\bigcap\limits_{i\in P}\sigma_i(X)=\emptyset$. Let us look at the map $f\ :\ i\longmapsto tp(\sigma_i(b)/AC)$. The image of $f$ is smaller than $2^{|AC|}$, so we can apply the pigeonhole principle to find $i_0<i_1<i_2...$ so that, for all $n<\omega$, $f(i_{n+1})=f(i_n)$. Let $C'=\sigma_{i_0}^{-1}(C)\equiv_A C$. For each $n<\omega$, $\sigma_{i_0}^{-1}\circ\sigma_{i_n}(b)$ is a $\sigma_{i_0}^{-1}(AC)=AC'$-conjugate of $b$, so we can use strong homogeneity to find $\tau_n\in Aut(\mathcal{M}/AC')$ ($\tau_0=id$) so that $\tau_n(b)=\sigma_{i_o}^{-1}\circ\sigma_{i_n}(b)$, ie $\tau_n(X)=\sigma_{i_o}^{-1}\circ\sigma_{i_n}(X)$. Then $N$, $(\tau_n)_{n<\omega}$ is a witness for division of $X$ over $AC'$.
\end{proof}

\subsection{Extension bases}

\begin{definition}
Let $A$ be a small subset of $\mathcal{M}$. Then $A$ is an \textit{extension base} if, for every small subset $C$ of $\mathcal{M}$, we have $\indep{C}{A}{A}{f}$.
\end{definition}

\begin{remark}
Note that models are extension bases: any type over a model admits a global coheir, which is in particular a global non-forking extension. In any elementary extension of the ordered group $\mathbb{Z}$, any definably-closed parameter set is a model, therefore any set in the theory of $(\mathbb{Z}, +, <)$ is an extension base.
\end{remark}

\begin{lemma}\label{transBExt}
Let $\mathfrak{A}$ be a large subset of $\mathcal{M}$, and $\mathfrak{C}$ be a class of small subsets of $\mathcal{M}$. Suppose that, for every singleton $c\in\mathfrak{A}$, for all $B\in\mathfrak{C}$, we have $Bc\in\mathfrak{C}$. Then the following conditions are equivalent:

\begin{itemize}
\item For every $B\in \mathfrak{C}$, for every  small subset $C$ of $\mathfrak{A}$, we have $\indep{C}{B}{B}{f}$ %for every small subset $C$ of $\mathcal{A}$, and $A\in\mathcal{C}$.
\item For every $B\in \mathfrak{C}$, for every singleton $c\in\mathfrak{A}$, we have $\indep{c}{B}{B}{f}$.
\end{itemize}
\end{lemma}

\begin{proof}
The first condition clearly implies the second one.
\par Suppose the first condition fails. Let $C$ be a small subset of $\mathfrak{A}$ so that $\notIndep{C}{B}{B}{f}$. By left finite character, there exist singletons $c_1...c_n\in C$ so that $\notIndep{c_1...c_n}{B}{B}{f}$. By (the contraposite of) left transitivity, there must exist $i$ so that $\notIndep{c_{i+1}}{Bc_1...c_i}{B}{f}$, so $\notIndep{c_{i+1}}{Bc_1...c_i}{Bc_1...c_i}{f}$ by definition. However, by hypothesis, we have $Bc_1...c_i\in\mathfrak{C}$, so the second condition fails, and we get the equivalence.
\end{proof}

\begin{lemma}\label{reducK1}
Suppose $\mathcal{M}$ is a Henselian valued field of residue characteristic zero. Let $A$ be a small subset of $\mathcal{M}^{eq}$. Suppose $\indep{C}{A}{A}{f}$ for every small subset $C$ of $K(\mathcal{M})$. Then $A$ is an extension base.
\end{lemma}

\begin{proof}
Suppose $\mathcal{M}$ is a Henselian valued field of residue characteristic zero. Let $C$ be a small subset of $\mathcal{M}^{eq}$. For each $c\in C$, the sort of $\mathcal{M}^{eq}$ that contains $c$ is the image of $K$ under some surjective $\emptyset$-definable function $f_c$ (in fact $\mathcal{M}^{eq}=\dcl^{eq}(K(\mathcal{M}))$), so we can find $d_c\in K(\mathcal{M})$ so that $f_c(d_c)=c$. Let $D$ be the reunion of all the $d_c$. Then $D$ is a small subset of $K(\mathcal{M})$ for which $C\subset\dcl(D)$, so $\acl(AC)\subset\acl(AD)$. By hypothesis, we have $\indep{D}{A}{A}{f}$, so we get $\indep{C}{A}{A}{f}$ by using \ref{leftAcl}. This holds for every $C$, so $A$ is an extension base.
\end{proof}

Note that the above lemma applies more generally to every first-order structure, where the sort $K$ would be replaced by the reunion of the home sorts of the structure at hand.

\begin{corollary}\label{reducK2}
Let $\mathfrak{C}$ be a class of small subsets of $\mathcal{M}$ for which, for every $A\in\mathfrak{C}$, $c\in K(\mathcal{M})$, we have $Ac\in \mathfrak{C}$. Then the following conditions are equivalent:
\begin{itemize}
\item Every element of $\mathfrak{C}$ is an extension base.
\item For every $A\in\mathfrak{C}$, for every singleton $c\in K(\mathcal{M})$, we have $\indep{c}{A}{A}{f}$.
\end{itemize}
\end{corollary}

\begin{proof}
This is an immediate consequence of \ref{reducK1} and \ref{transBExt} (with $\mathfrak{A}=K(\mathcal{M})$).
\end{proof}

\begin{definition}
Suppose $\mathcal{M}$ is a field. We denote classical field algebraic independence by $\indep{C}{A}{B}{alg}$, that is, for every finite tuple $c$ from the field generated by $AC$, the transcendence degree of $c$ over the field generated by $A$ equals its transcendence degree over the field generated by $AB$.
\end{definition}

Instead of checking the definition for every finite tuple, one can equivalently check equality of the transcendence degrees for every finite tuple that is algebraically independent over the field generated by $A$.

\begin{fact}
If $\mathcal{M}$ is a field, then $\indep{}{}{}{f}\subset\indep{}{}{}{alg}$.
\end{fact}

%\begin{lemma}
%If $\mathcal{M}$ is a field, then the following conditions are equivalent:
%
%\begin{itemize}
%\item Every complete unary type over $\acl(AB)$ of an $\acl(AB)$-transcendental element does not fork over $A$.
%\item $\indep{}{}{}{f}=\indep{}{}{}{alg}$.
%\end{itemize}
%\end{lemma}
%
%%The first condition looks weaker than the second one at first. We use it to prove our main results \ref{introMain2}, \ref{introMain3}, and we first thought this hypothesis could be weaker than $C^3_k$, which would yield more general results. We quickly realized that these two conditions are actually equivalent.
%
%\begin{proof}
%By \ref{rightAcl}, the second condition implies the first one.
%\par Suppose the first condition holds. Let $A$, $B$, $C$ be small subsets of $\mathcal{M}$ for which $\indep{C}{A}{B}{alg}$. We just have to prove that $\indep{C}{A}{B}{f}$. By left finite character, we can assume $C=c_1...c_n$ is finite. Let $d_1...d_m\in\mathcal{M}$ be a transcendance base of the field $\acl(Ac_1...c_n)$ over the field $\acl(A)$. By hypothesis, we have $\indep{d_{i+1}}{Ad_1...d_i}{\acl(ABd_1...d_i)}{f}$ for every $i$, so in particular $\indep{d_{i+1}}{Ad_1...d_i}{B}{f}$. By left transitivity, we have $\indep{d_1...d_m}{A}{B}{f}$. As $\acl(Ac_1...c_n)\subset\acl(Ad_1...d_m)$, we can conclude by using \ref{leftAcl}.
%\end{proof}

\begin{remark}
If $\mathcal{M}$ is a field, and $\indep{}{}{}{f}=\indep{}{}{}{alg}$, then it is easy to observe that every set is an extension base.
\end{remark}

%!TEX root = main.tex

\section{Main results}\label{sectMain}

In this section, we fix $\kappa>\aleph_0$, $\mathcal{M}$ a $\kappa$-saturated and $\kappa$-strongly homogeneous Henselian valued field of residue characteristic zero.

\par In order to study extension bases, we will use the argument of \ref{reducK2} to reduce to the case of unary types. We will start this section by giving sufficient conditions for a global unary type to be non-forking over some parameter set $A$. We have three distinct cases that correspond to the three following lemmas. The cases depend on the nature of the chain of $A$-balls of which our unary type is $A$-generic, as well as whether $A$ is generated by parameters from $K$.
\par The following lemma deals with the case of a non-residual chain, with $A$ not necessarily generated by field parameters.

\begin{lemma}\label{indepGamma}
Let $A$ be a small subset of $\mathcal{M}^{eq}$, $\mathfrak{B}$ a chain of $A$-balls, $B$ a $(2^{|A|})^{++}$-saturated, strongly-$(2^{|A|})^{++}$-homogeneous small elementary substructure of $\mathcal{M}^{eq}$ that contains $A$, and $c\in K(\mathcal{M})$. Assume the following conditions hold:
\begin{itemize}
\item $A=\acl^{eq}(A)$.
\item $\mathfrak{B}$ is not residual.
%\item $b\in\cap\mathfrak{B}$.
\item $c$ is $B$-generic of $\mathfrak{B}$.
\item There exists $b\in K(B)\cap\left(\cap\mathfrak{B}\right)$ so that $\indep{val(c-b)}{\Gamma(A)}{\Gamma(B)}{f}$ in the ordered group structure of $\Gamma(\mathcal{M})$.
\item For every $N<\omega$, $[k^* : (k^*)^N]$ is finite.
\end{itemize}
Then we have $\indep{c}{A}{B}{f}$.
\end{lemma}

\begin{proof}
First of all, let us notice that the smallest closed ball that contains $b$ and $c$ is contained in $\cap\mathfrak{B}$, and this inclusion is strict, otherwise this closed ball would be the least element of $\mathfrak{B}$, contradicting the hypothesis. By genericity of $c$, this ball does not belong to $B$, but it has a point $b$ in $B$, so its radius $val(c-b)$ must not be in $\Gamma(B)$.
\par Suppose by contradiction that $\notIndep{c}{A}{B}{f}$. Then, by \ref{forkingEZ}, we have $\notIndep{c}{A}{B}{d}$. By \ref{zetaBall} ($B=\acl^{eq}(K(B))$), there exists $X$ a unary $B$-definable subset of $\RV$ that contains $rv (c-b)$ so that $X'=\left\lbrace x\in K|rv(x-b)\in X\right\rbrace$ divides over $A$. By \ref{divisionBase}, there exists $b'\in B$ so that $b\equiv_A b'$ and $X'$ divides over $Ab'$. We have $val(c-b)<val(b-b')$, so $rv(c-b)=rv(c-b')$. Thus, by replacing $X$ by $X\cap val^{-1}(]-\infty, val(b-b')[)\ni rv(c-b)$, we have $rv(x-b)\in X\Longleftrightarrow rv(x-b')\in X$ for all $x\in K$, and we can suppose $b'=b$ without loss of generality. One can notice that a witness for division of $X'$ over $Ab'$ is also a witness for division of the translate $X'-b'$ over $A$, so we can suppose $b=b'=0$ without loss of generality (so $Ab'=A$). By \ref{divisionPreimage}, we can deduce that $X$ divides over $A$, as $X'$ is the preimage of $X$ under the $A$-definable function $rv$. Now, let $(\gamma_i)_i$, $(X_i)_i$, $N<\omega$ be a witness of Lemma \ref{zetaRV} applied to $X$. The family $(\gamma_i)_i$ is finite, so we can assume that the $\gamma_i$ are algebraic over $B$, ie they belong to $\Gamma(B)$, so they are all distinct from $val(c)$. As a result, we can replace $X$ by $X\setminus\left(\bigcup\limits_i X_i\right)$ without loss of generality. As $X$ is invariant under multiplication by $(k^*)^N$, $X$ is the preimage of $\bar{X}=X \mod (k^*)^N$, so $\bar{X}$ divides over $A$. The $A$-definable group homomorphism $val\ :\ \faktor{\RV}{(k^*)^N}\longrightarrow \Gamma$ has a finite kernel $\faktor{k^*}{(k^*)^N}$. We can apply \ref{fibresFinies} to the relation:
$$
R=\left\lbrace (x, y)\in \Gamma\times \faktor{\RV}{(k^*)^N}|val(y)=x\right\rbrace
$$
to show that $val(\bar{X})$ divides over $A$. Now we are done: take a witness for division $(\sigma_n)_n$ of $val(\bar{X})$ over $A$, $({\sigma_n}_{|\Gamma(\mathcal{M})})_n$ is a witness for division of $val(\bar{X})$ over $\Gamma(A)$ in the ordered group $\Gamma(\mathcal{M})$ ; we know that $val(\bar{X})$ is $\Gamma(B)$-definable by \ref{stablementPlonge}, and $val(c)=val(c-b)\in val(\bar{X})$, so $\notIndep{val(c-b)}{\Gamma(A)}{\Gamma(B)}{d}$, a contradiction.
\end{proof}

\begin{remark}
In the above lemma, the notions of $B$-genericity and $B$-weak-genericity are the same: $B$ is a model, so every $B$-ball is pointed.
\end{remark}

\par The following lemma deals with the case of a residual chain, with $A$ not necessarily generated by field parameters.

\begin{lemma}\label{indepKImaginaire}
Let $A$ be a small subset of $\mathcal{M}^{eq}$, $Z$ a closed $A$-ball, $B$ a $(2^{|A|})^{++}$-saturated, strongly-$(2^{|A|})^{++}$-homogeneous small elementary substructure of $\mathcal{M}^{eq}$ that contains $A$, and $c\in K(\mathcal{M})$. Assume the following conditions hold:
\begin{itemize}
\item $A=\acl^{eq}(A)$.
\item $c$ is $B$-generic of $\{Z\}$.
\item In the field $k(\mathcal{M})$, we have $\indep{}{}{}{f}=\indep{}{}{}{alg}$.
\end{itemize}
Then we have $\indep{c}{A}{B}{f}$.
\end{lemma}

\begin{proof}
Suppose towards contradiction that $\notIndep{c}{A}{B}{d}$. Let $X$, $X'$ be as in the proof of \ref{indepGamma}. We can assume $X'\subset Z$. By \ref{divisionBase}, there exist $b$, $b'\in B$ so that $b\in Z$, $val(b')=rad(Z)$, and $X'$ divides over $Abb'$. Just like in the previous proof, a witness for division of $X'$ over $Abb'$ is a witness for division of $\frac{X'-b}{b'}$ over $A$, so we can assume $b=0$, $b'=1$, and $Z=\mathcal{O}$. By genericity of $c$, and by \ref{zetaBall}, we know that, for all $c'\in K(\mathcal{M})$, if $res(c')=res(c)$, then $c\equiv_B c'$. As a result, the set:
$$
\left\lbrace x\in X'|\forall y\in \mathcal{O}\ \left( res(x)=res(y)\Longrightarrow y\in X'\right)\right\rbrace
$$
is a $B$-definable subset of $X'$ containing $c$, so it still divides over $A$, and we can assume that $X'$ coincides with this set. Now $X'$ is the preimage of $res(X')$, so by \ref{divisionPreimage} $res(X')$ divides over $A$. By \ref{stablementPlonge} we have $\notIndep{res(c)}{k(A)}{k(B)}{f}$ in the field $k$, so by hypothesis $\notIndep{res(c)}{k(A)}{k(B)}{alg}$, and $res(c)$ is not transcendental over $k(B)$. This contradicts the genericity of $c$.
\end{proof}

In the next lemma, we also deal with the residual case, but without the assumption $\indep{}{}{}{f}=\indep{}{}{}{alg}$ in $k(\mathcal{M})$. However, we have to strengthen our hypothesis on $A$ to make it work, and suppose that $A$ is generated by field elements.

\begin{lemma}\label{indepKReel}
Let $A$ be a small subset of $\mathcal{M}^{eq}$, $Z$ a closed $A$-ball, $B$ a $(2^{|A|})^{++}$-saturated, strongly-$(2^{|A|})^{++}$-homogeneous small elementary substructure of $\mathcal{M}^{eq}$ that contains $A$, and $c\in K(\mathcal{M})$. Assume the following conditions hold:
\begin{itemize}
\item $A=\acl^{eq}(K(A))$.
\item $c$ is $B$-generic of $\{Z\}$.
\item There exist $b\in Z(A)$ and $b'\in K(A)$ so that $val(b')=rad(Z)$, and, in the pure field $k(\mathcal{M})$, we have $\indep{res\left(\frac{c-b}{b'}\right)}{k(A)}{k(B)}{f}$.
\end{itemize}
Then we have $\indep{c}{A}{B}{f}$.
\end{lemma}

\begin{proof}
Let $c'=\frac{c-b}{b'}$. It is clear that $c'$ is $B$-generic of $\mathcal{O}$. As $\indep{res\left(\frac{c-b}{b'}\right)}{k(A)}{k(B)}{f}$ in the field $k(\mathcal{M})$, we have $\indep{c'}{A}{B}{f}$ with a reasoning similar to the previous lemma. Now, this time $b$, $b'\in K(A)$, so $c\in\acl(Ac')$, so $\indep{c}{A}{B}{f}$ by \ref{leftAcl}.
\end{proof}

Now that we have seen several sufficient conditions for a global unary type to be non-forking over $A$, we will need to be able to build these types, ie realize these conditions. It will be easy in some cases, and the construction will be described quickly in the proofs of our theorems. The two following lemmas deal with cases where a more technical approach is required. The next lemma will be used to build a non-forking global extension of the type of an $A$-generic point of an $A$-immediate chain. The second lemma does the same for the generic of an $A$-ramified chain, but there we will have to assume that $A$ is generated by field elements.

\begin{remark}
In the residual case, the non-forking global extension is easy to build when $\indep{}{}{}{f}=\indep{}{}{}{alg}$ in $k(\mathcal{M})$: Lemma \ref{indepKImaginaire} merely requires $c$ to be $B$-generic of the correct chain. One can note that this lemma holds with a weaker hypothesis: instead of requiring $\indep{}{}{}{f}=\indep{}{}{}{alg}$ in $k(\mathcal{M})$, we can just suppose that for all $b$, $b'\in K(B)$, if $b\in Z$ and $val(b')=rad(Z)$, then we have $\indep{res\left(\frac{c-b}{b'}\right)}{k(A)}{k(B)}{f}$ in the field structure of $k(\mathcal{M})$. However, without the hypothesis $\indep{}{}{}{f}=\indep{}{}{}{alg}$ in $k(\mathcal{M})$, it is not clear whether the existence of the non-forking global extension can be established, because a random $B$-generic point of $\{Z\}$ might not satisfy these new conditions. Maybe such a point still exists, this is one of the interesting questions that naturally follow from this paper.
\end{remark}

\begin{lemma}\label{immediat}
Let $A$ be a small subset of $\mathcal{M}^{eq}$. Let $\mathfrak{B}$ be an non-residual chain of $A$-balls. Let $B$ be a $(2^{|A|})^{++}$-saturated, strongly-$(2^{|A|})^{++}$-homogeneous small elementary substructure of $\mathcal{M}^{eq}$ that contains $A$, and $b\in K(B)$ so that $b\in\cap\mathfrak{B}$. Suppose the following conditions hold:
\begin{itemize}
\item $A=\acl^{eq}(A)$.
\item In the ordered group $\Gamma(\mathcal{M})$, $\Gamma(A)$ is an extension base.
\end{itemize}
Then there exists $c\in K(\mathcal{M})$ so that $c$ is $B$-generic of $\mathfrak{B}$, and $\indep{val(c-b)}{\Gamma(A)}{\Gamma(B)}{f}$ in the ordered group structure.
\end{lemma}

\begin{proof}
Define:
$$\bar{A}=\left\lbrace rad(X)|X\in\mathfrak{B}\right\rbrace\subset\Gamma(A)$$
$$\bar{B}=\left\lbrace \delta\in\Gamma(B)|\forall \gamma\in\bar{A}\ \gamma<\delta\right\rbrace$$
Let $X=\bigcap\limits_{\gamma\in\bar{A}, \delta\in\bar{B}}]\gamma, \delta[$, $(\sigma_n)_n\in Aut(\Gamma(\mathcal{M})/\Gamma(A))^\omega$, and $\bar{B}'=\bigcup\limits_n \sigma_n(\bar{B})$. Then $\bar{A}$ and $\bar{B}'$ satisfy the hypothesis of \ref{coupure}, so the type-definable set $Y=\bigcap\limits_{\gamma\in\bar{A}, \delta\in\bar{B}'}]\gamma, \delta[$ is non-empty. However, we have $Y=\bigcap\limits_n \sigma_n(X)$, so $(\sigma_n)_n$ cannot be a witness for division of (any $\Gamma(B)$-definable set containing) $X$ over $\Gamma(A)$. Therefore, $X$ does not divide over $\Gamma(A)$ in the ordered group $\Gamma(\mathcal{M})$. In this ordered group, the induced theory is dependent (\cite{NIPOAG}), and $\Gamma(A)$ is an extension base, so forking over $\Gamma(A)$ coincides with dividing over $\Gamma(A)$ (\cite{KapCherNTP2}, Theorem 1.2). As a result, $X$ does not fork over $\Gamma(A)$ in $\Gamma(\mathcal{M})$. By \ref{bonnesExtensions}, there exists $\gamma\in\Gamma(\mathcal{M})$ so that $\gamma\in X$, and $\indep{\gamma}{\Gamma(A)}{\Gamma(B)}{f}$.
\par Now we are done: a witness of the lemma will be any $c\in K(\mathcal{M})$ so that $val(c-b)=\gamma$. Such a point $c$ must be $B$-generic of $\mathfrak{B}$: $val(c-b)>\bar{A}$ implies $c\in\cap\mathfrak{B}$, and $val(c-b)<\bar{B}$ implies that $c$ does not belong to any smaller $B$-ball. 
\end{proof}

\begin{lemma}\label{ramifieReel}
Let $A$ be a small subset of $\mathcal{M}^{eq}$. Let $\mathfrak{B}$ be an $A$-ramified chain of $A$-balls. Let $B$ be a $(2^{|A|})^{++}$-saturated, strongly-$(2^{|A|})^{++}$-homogeneous small elementary substructure of $\mathcal{M}^{eq}$ that contains $A$, $c\in K(\mathcal{M})$, and $b\in K(A)$ so that $b\in\cap\mathfrak{B}$. Suppose the following conditions hold:
\begin{itemize}
\item $A=\acl^{eq}(K(A))$.
\item $c$ is $A$-generic of $\mathfrak{B}$.
\item In the ordered group $\Gamma(\mathcal{M})$, $\Gamma(A)$ is an extension base.
\end{itemize}
Then there exists $c'\equiv_A c$ so that $c'$ is $B$-generic of $\mathfrak{B}$, and $\indep{val(c'-b)}{\Gamma(A)}{\Gamma(B)}{f}$ in the ordered group structure.
\end{lemma}

\begin{proof}
Let $\bar{A}$, $\bar{B}$, $X$, $(\sigma_n)_n$, $\bar{B}'$, $Y$ be as in the proof of the previous lemma. Let $B'$ be a small $|B|^+$-saturated elementary substructure of $\mathcal{M}$ containing $B\cup\bar{B}'$. By \ref{consistanceGenerique}, there exists $d\equiv_A c$ so that $d$ is $B'$-generic of $\mathfrak{B}$. The value $val(d-b)$ not only is a point of $Y$, put it is also a $\Gamma(A)$-conjugate of $val(c-b)$. As a result, the partial type (of the ordered group) $\tp(val(c-b)/\Gamma(A))\cup\{x\in X\}$ does not divide, and hence does not fork over $\Gamma(A)$. Let $\gamma\in\Gamma(\mathcal{M})$ be a realization of this type for which $\indep{\gamma}{\Gamma(A)}{\Gamma(B)}{f}$. We have $\gamma\equiv_{\Gamma(A)} val(c-b)$, so $\gamma\equiv_A val(c-b)$ by \ref{dominationValeurCorpique}. Let $\sigma\in Aut(\mathcal{M}/A)$ so that $\sigma(val(c-b))=\gamma$. As $b\in K(A)$, $\sigma(b)=b$, so $c'=\sigma(c)$ witnesses the lemma.
\end{proof}

We can finally prove our main results:

\begin{theorem}\label{main1}
Let $A$ be a small subset of $K(\mathcal{M})$. Suppose the following conditions hold:
\begin{itemize}
\item Every subset of $\Gamma(\mathcal{M})$ is an extension base with respect to the ordered group structure.
\item $\Gamma(\mathcal{M})$ is dense.
\item Every subset of $k(\mathcal{M})$ is an extension base with respect to the field structure.
\item For every $N<\omega$, $[k^* : (k^*)^N]$ is finite.
\item For every $\gamma\in\Gamma(\dcl(\emptyset))$, there exists an $a\in K(\acl(A))$ so that $val(a)=\gamma$.
\item For every $N>0$, $\alpha\in\faktor{k^*}{(k^*)^N}(\mathcal{M})$, there exists $a\in \mathcal{O}^*_{\acl(A)}$ so that \\$res(a)\mod (k^*)^N=\alpha$.
\end{itemize}
Then $A$ is an extension base.
\end{theorem}

\begin{proof}
Let $\mathfrak{C}$ be the class of every small subsets $A'$ of $K(\mathcal{M})$ satisfying the two last points of the theorem, and such that $\Gamma(A')$ is an extension base in $\Gamma(\mathcal{M})$. Clearly, for all $A'\in\mathfrak{C}$, for all $c\in K(\mathcal{M})$, we have $A'c\in\mathfrak{C}$.\footnote{There is a subtlety here. The $\faktor{k^*}{(k^*)^N}$ are fixed. If they were infinite (hence unbounded), maybe adding $c$ to $A'$ would add new classes of $(k^*)^N$ that cannot be pulled back to $K(A'c)$! So the finiteness hypothesis is crucial for the induction argument of \ref{reducK2}.} By \ref{reducK2}, we just have to show that $\indep{c}{A}{A}{f}$ for every singleton $c\in K(\mathcal{M})$.
\par Let $c\in K(\mathcal{M})$. As $\indep{c}{A}{A}{f}\Longleftrightarrow\indep{c}{\acl^{eq}(A)}{\acl^{eq}(A)}{f}$, we can replace $A$ by $\acl^{eq}(A)$. Let $B$ be a $(2^{|A|})^{++}$-saturated, strongly-$(2^{|A|})^{++}$-homogeneous small elementary substructure of $\mathcal{M}^{eq}$ that contains $A$. We have to find $c'\equiv_A c$ for which $\indep{c'}{A}{B}{f}$. Let $\mathfrak{B}$ be the set of every $A$-ball containing $c$.
\begin{itemize}
\item Suppose $\mathfrak{B}$ is $A$-immediate. Let $c'$ witness Lemma \ref{immediat}. We have $\indep{c'}{A}{B}{f}$ by \ref{indepGamma}. Let us show $c\equiv_A c'$. Let $\mathfrak{B}'$ be the chain of every pointed $A$-ball that contains $c$. If $\cap\mathfrak{B}'$ has no point in $K(A)$, then we have $c\equiv_A c'$ by \ref{zetaBall}, as $c'$ is clearly in $\cap\mathfrak{B}'$. Else, if $a\in K(A)$ is in $\cap\mathfrak{B}'$, then $\mathfrak{B}'$ is not $A$-immediate, so $\cap\mathfrak{B}'\neq\cap\mathfrak{B}$, and $c$ is not $A$-generic of $\mathfrak{B}'$. By \ref{fortVSFaible}, $\mathfrak{B}$ has a least element $Y$, an open ball which does not contain $a$. As $c$ and $c'$ are both in $Y$, we clearly have $rv(c-a)=rv(c'-a)$, which also implies $c\equiv_A c'$ by \ref{zetaBall}.
\item Suppose $\mathfrak{B}$ is $A$-ramified, ie $\cap\mathfrak{B}$ has a point $b\in K(A)$. Let $c'$ witness Lemma \ref{ramifieReel} applied to $b$. We have $\indep{c'}{A}{B}{f}$ by \ref{indepGamma}.
\item Suppose $\mathfrak{B}$ is residual, let $X=min(\mathfrak{B})$. Then $X$ is pointed by \ref{relevementCentre}. Let $b\in K(A)$ so that $b\in X$. By \ref{relevementRayon2}, there also exists $b'\in K(A)$ so that $val(b')=rad(X)$. As $k(A)$ is an extension base in the residue field, we can use \ref{bonnesExtensions} to find $\alpha\equiv_{k(A)}res\left(\frac{c-b}{b'}\right)$ so that $\indep{\alpha}{k(A)}{k(B)}{f}$. By \ref{dominationResiduCorpique}, we have $\alpha\equiv_A res\left(\frac{c-b}{b'}\right)$, so there exists $c'\equiv_A c$ for which $\indep{res\left(\frac{c'-b}{b'}\right)}{k(A)}{k(B)}{f}$. By \ref{indepKReel}, we have $\indep{c'}{A}{B}{f}$.\qedhere
\end{itemize}
\end{proof}

\begin{theorem}\label{main1Bis}
The above theorem also holds if we replace the hypothesis on the density of $\Gamma$ by the following condition:
\par In the pure field $k(\mathcal{M})$, for any subfield $B$ for which $B^{alg}\cap k(\mathcal{M})=B$, we have $B=\acl(B)$ in the theory of the field $k(\mathcal{M})$.
\end{theorem}

\begin{proof}
The only difference in the proof is in the residual case, where we have to use \ref{relevementCentre2} instead of \ref{relevementCentre} to show that $X$ is pointed.
\end{proof}

\begin{theorem}\label{main2}
Let $A$ be a small subset of $K(\mathcal{M})$. Suppose the following conditions hold:
\begin{itemize}
\item Every subset of $\Gamma(\mathcal{M})$ is an extension base with respect to the ordered group structure.
\item $\indep{}{}{}{f}=\indep{}{}{}{alg}$ in the field $k(\mathcal{M})$.
\item For every $N<\omega$, $[k^* : (k^*)^N]$ is finite.
\end{itemize}
Then $A$ is an extension base.
\end{theorem}

\begin{proof}
We do the same proof as above with $\mathfrak{C}$ the class of every small subset of $K(\mathcal{M})$. The only case where we have to act differently is the residual case. Here, we use \ref{consistanceGenerique} to find $c'\equiv_A c$ which is $B$-generic of $\mathfrak{B}$. Then $\indep{c'}{A}{B}{f}$ by \ref{indepKImaginaire}.
\end{proof}

\begin{theorem}\label{main3}
Let $A$ be a small subset of $\mathcal{M}^{eq}$. Suppose the following conditions hold:
\begin{itemize}
\item $\Gamma(\mathcal{M})\equiv \mathbb{Z}$.
\item $\indep{}{}{}{f}=\indep{}{}{}{alg}$ in the field $k(\mathcal{M})$.
\item For every $N<\omega$, $[k^* : (k^*)^N]$ is finite.
\end{itemize}
Then $A$ is an extension base.
\end{theorem}

\begin{proof}
This time $\mathfrak{C}$ is the class of every small subset of $\mathcal{M}^{eq}$. We have the same disjunction of three cases:
\begin{itemize}
\item $\mathfrak{B}$ is $A$-immediate. We proceed just like the proof of \ref{main1}, except that we have to use \ref{zetaBall2} instead of \ref{zetaBall} to prove $c'\equiv_A c$.
\item $\mathfrak{B}$ is $A$-ramified, ie $\cap\mathfrak{B}$ contains an $A$-ball $X$. By \ref{consistanceGenerique}, let $c'\equiv_A c$ a $B$-generic point of $\mathfrak{B}$. By genericity of $c'$, we can apply \ref{coupure} to show that $\indep{val(c-X)}{\Gamma(A)}{\Gamma(B)}{f}$ in the ordered group $\Gamma(\mathcal{M})$. Then we can apply \ref{indepGamma} (with $b$ any point in $K(B)\cap X$) to get $\indep{c}{A}{B}{f}$.
\item $\mathfrak{B}$ is residual. This time, we proceed like the proof of \ref{main2}.\qedhere
\end{itemize}
\end{proof}

%\pagebreak
\tableofcontents
%\pagebreak
\bibliographystyle{plain}
\bibliography{biblio}

\begin{thebibliography}{10}

\bibitem{suitesExactes}
Matthias Aschenbrenner, Artem Chernikov, Allen Gehret, and Martin Ziegler.
\newblock Distality in valued fields and related structures.
\newblock {\em Transactions of The American Mathematical Society - TRANS AMER
  MATH SOC}, 375(7):4641 -- 4710, 2022.

\bibitem{yaacov_chernikov_2014}
Itaï Ben~Yaacov and Artem Chernikov.
\newblock An independence theorem for ntp2 theories.
\newblock {\em The Journal of Symbolic Logic}, 79(1):135–153, 2014.

\bibitem{Casanovas2011SimpleTA}
Enrique Casanovas.
\newblock {\em Simple Theories and Hyperimaginaries}.
\newblock Cambridge University Press, March 2011.

\bibitem{Cherlin1976ModelTA}
Gregory~L. Cherlin.
\newblock {\em Model theoretic algebra: Selected topics}.
\newblock Springer Berlin, Heidelberg, 1976.

\bibitem{CHERNIKOV2014695}
Artem Chernikov.
\newblock Theories without the tree property of the second kind.
\newblock {\em Annals of Pure and Applied Logic}, 165(2):695 -- 723, February
  2014.

\bibitem{KapCherNTP2}
Artem Chernikov and Itay Kaplan.
\newblock Forking and dividing in ntp2 theories.
\newblock {\em J. Symb. Log.}, 77(1):1 -- 20, March 2012.

\bibitem{Chernikov2016HENSELIANVF}
Artem Chernikov and Pierre Simon.
\newblock Henselian valued fields and inp-minimality.
\newblock {\em The Journal of Symbolic Logic}, 84:1510 -- 1526, 2016.

\bibitem{QEOAG}
Raf Cluckers and Immanuel Halupczok.
\newblock Quantifier elimination in ordered abelian groups.
\newblock {\em Confluentes Mathematici}, 03(4):587 -- 615, October 2011.

\bibitem{ealy2022residue}
Clifton Ealy, Deirdre Haskell, and Pierre Simon.
\newblock Residue field domination in some henselian valued fields, 2022.

\bibitem{flenner}
Joseph Flenner.
\newblock {Relative decidability and definability in henselian valued fields}.
\newblock {\em Journal of Symbolic Logic}, 76(4):1240 -- 1260, December 2011.

\bibitem{NIPOAG}
Yuri Gurevich and P.~Schmitt.
\newblock The theory of ordered abelian groups does not have the independence
  property.
\newblock {\em Transactions of The American Mathematical Society - TRANS AMER
  MATH SOC}, 284(1):171 -- 182, July 1984.

\bibitem{HHM}
Deirdre Haskell, Ehud Hrushovski, and Dugald Macpherson.
\newblock {\em Stable domination and independence in algebraically closed
  valued fields}.
\newblock Cambridge University Press, December 2005.

\bibitem{zeta}
Ehud Hrushovski, Ben Martin, Silvain Rideau, and Raf Cluckers.
\newblock Definable equivalence relations and zeta functions of groups.
\newblock {\em Journal of the European Mathematical Society}, 20(10):2467 --
  2537, July 2018.

\bibitem{hrushovskiPillay}
Ehud Hrushovski and Anand Pillay.
\newblock On nip and invariant measures.
\newblock {\em Journal of the European Mathematical Society}, 13, 11 2007.

\bibitem{kimPillay}
Byunghan Kim and Anand Pillay.
\newblock From stability to simplicity.
\newblock {\em Bulletin of Symbolic Logic}, 4(1):17 -- 36, March 1998.

\bibitem{KUHLMANN2022339}
Salma Kuhlmann and Michele Serra.
\newblock The automorphism group of a valued field of generalised formal power
  series.
\newblock {\em Journal of Algebra}, 605:339 -- 376, September 2022.

\bibitem{Tent2012ACI}
Katrin Tent and Martin Ziegler.
\newblock {\em A Course in Model Theory}.
\newblock Cambridge University Press, June 2012.

\bibitem{vandenDries2014}
Lou van~den Dries.
\newblock Lectures on the model theory of valued fields.
\newblock In {\em Model Theory in Algebra, Analysis and Arithmetic: Cetraro,
  Italy 2012, Editors: H. Dugald Macpherson, Carlo Toffalori}, pages 55 -- 157.
  Springer Berlin Heidelberg, January 2014.

\end{thebibliography}

\end{document}